\documentclass[twoside,a4paper,11pt,limits]{amsart}

\usepackage{amsmath}
\usepackage{amsthm}
\usepackage{amssymb}
\usepackage{amsfonts}
\usepackage{amsxtra}
\usepackage{comment}

\usepackage{enumitem}
\setenumerate{label={\rm (\alph{*})}}

\usepackage[frak,operator]{paper_diening}
\usepackage{graphicx}
\newcommand{\local}{{\text{loc}}}
\newtheorem{theorem}{Theorem}[section]
\newtheorem{lemma}[theorem]{Lemma}
\newtheorem{corollary}[theorem]{Corollary}
\newtheorem{proposition}[theorem]{Proposition}
\newtheorem{remark}[theorem]{Remark}

\newtheorem{assumption}[theorem]{Assumption}

\numberwithin{equation}{section}

\newcommand{\meantmp}[2]{#1\langle{#2}#1\rangle}

\newcommand{\RNn}{\setR^{N\times n}}

\newcommand{\dint}{{\dashint}}

\newcommand{\unite}{{\!\!\!\!\!\!\!\!\!\!\!\!\!}}
\newcommand{\nsk}{{\!\!\!}}

\newcommand{\slant}{{\nearrow}}

\providecommand{\da}{\,{d}a}
\providecommand{\dx}{\,{d}x}

\providecommand{\dz}{\,{d}z}
\providecommand{\ds}{\,{d}s}
\providecommand{\dt}{\,{d}t}

\providecommand{\Div}{\divergence}

\begin{document}

\title{On regularity of the time derivative for degenerate parabolic systems}
\thanks{Sebastian Schwarzacher is financed by the program PRVOUK P47 at the department of analysis at the Charles University Prag.}
\author{Jens Frehse}
\address{Jens Frehse,
Department of applied analysis,
University of Bonn,
Endenicheralle 60,
53115 Bonn, Germany}
\email{erdbeere@iam.uni-bonn.de}
\author{Sebastian~Schwarzacher}
\address{Sebastian Schwarzacher, Department of mathematical analysis, Faculty of Mathematics and Physics,  Charles University in Prague,
Sokolovsk\'{a} 83, 186 75 Prague, Czech Republic}
\email{schwarz@karlin.mff.cuni.cz}

\begin{abstract}
  We prove regularity estimates for time derivatives of a large class of nonlinear parabolic partial differential systems. This includes the instationary (symmetric) p-Laplace system and models for non Newtonien fluids of powerlaw or Carreau type. By the use of special weak different quotients, adapted to the variational structure we bound fractional derivatives of $u_t$ in time and space direction. 

Although the estimates presented here are valid under very general assumptions they are a novelty even for the parabolic $p$-Laplace equation. 
\end{abstract}


\maketitle
\section{Introduction}


In this paper we will prove fractional differentiability for the time derivative $u_t$. This we can do for a very general class of essential non-linear PDE; the model case is the parabolic p-Laplace. But our results include models of non-Newtonian fluids of power-law type initiated by Ladyzhenskaya~\cite{Lad67,Lad68} and J.J.-Lions~\cite{Lio69} (see Subsection~\ref{ssec:nnfl} below). 

 To our knowledge little regularity for the time derivative is known up to now even for the homogeneous instationary p-Laplace equation (i.e. \eqref{eq:plap}, with $f\equiv0$). In comparison quite a lot of regularity is known for the space gradient. After the pioneering work of Friedman and DiBenedetto~\cite{DiBFri85} where they proof that space gradients are H\"older continous if the right hand side is zero a multifarious collection of results on gradient regularity was developed. We will discuss the known results and its implication of the regularity of $u_t$ more detailed in the next subsection.

In the framework of embedding theory, interpolation to finite element spaces or compactness issues the regularity of the time derivative is of essential importance. Indeed, its regularity often restricts the estimates crucially. Therefore and because of the generality of our approach, we emphasize that the results and techniques developed in this article will be an important step in the analysis and numeric to many instationary applications.
      
\subsection{The parabolic p-Laplace system}
The reader will be introduced to the results of this article by explaining them on the model case. Namely, the (symmetric) parabolic p-Laplace system. 
\begin{align}
 \label{eq:plap}
\begin{aligned}
 u_t-\divergence(\abs{D u}^{p-2}D u) &= f\text{ on } Q_T\\
u&=0\text{ on }(0,T)\times\partial \Omega\\
u(0,x)&=u_0(x)\text{ on }\Omega.
\end{aligned}
\end{align}
Here $\Omega\subset\setR^n$ is a domain in space and $[0,T)$ a time interval $Q_T:=[0,T)\times \Omega$. We will use $Dv$ as a substitute for both the gradient $\nabla v$ as well as the symmetric gradient $\epsilon v:=\frac12(\nabla v+(\nabla v)^T$ (which is of course only defined, whenever $n=N$). This is due to the fact, that we can treat both cases simultaneously.
Standard existence theory implies that there is a unique solution $u\in L^p((0,T),V^{1,p}_0(\Omega))\cap L^\infty([0,T),L^2(\Omega))$ for every right hand side $f\in L^{p'}([0,T),(V^{1,p}_0(\Omega))^*)$ and every $u_0\in L^2(\Omega)$. The space $V^{1,p}_0(\Omega):=\overline{C^\infty_0(\Omega)}^{\norm{\nabla \cdot}_{L^p(\Omega)}}$.

Our aim are estimates for $u_t$ in terms of Nikolskij spaces. These are spaces that represent fractional derivatives. For $\alpha\in (0,1]$ (the order of derivative) and $q\in [1,\infty)$ (the exponent of integrability) we say that $g\in \mathcal{N}^{\alpha,q}((a,b))$, if
\[ \norm{g}_{\mathcal{N}^{\alpha,q}((a,b))}^q:=\sup_{h>0}\int_a^{b-h}\Bigabs{\frac{g(t+h)-g(t)}{h^\alpha}}^q\dt+\int_a^{b}\abs{g}^q\dt<\infty. 
\]

For a general domain $A$ we say that $g\in \mathcal{N}^{\alpha,q}(A)$ if the property above holds in all coordinate directions.
Nikolskij spaces are closely related to fractional Sobolev spaces $W^{\alpha,q}(A)$.  

Let us just mention that $0<\alpha<\beta<1$ both
\[
W^{\beta,q}((a,b))\subset\mathcal{N}^{\beta,q}((a,b))\subset 
W^{\alpha,q}((a,b)).
\]
In our setting the relation can be used, as for fractional Sobolev spaces the following embedding theorem is available: For a Lipschitz domain $A\subset\setR^n$ and $g\in W^{\alpha,q}(A)$ we have 
\begin{align}
\label{eq:se}
 \norm{g}_{L^\frac{nq}{n-\alpha q}(A)}\leq c\norm{g}_{W^{\alpha,q}(A)}.
\end{align}
This implies for instance that if $g\in \mathcal{N}^{\alpha,q}(A)$, $g\in L^a_{\text{local}}(A)$ for every $1\leq a<\frac{nq}{n-\alpha q}$.
 For the embedding theorem and a more detailed study on the given function spaces we refer to~\cite{Tri92, Ada75}. For the study of fractional spaces in the framework of PDE we recommend~\cite{Min06}.

The first result of this article is that $u_t$ has fractional derivatives of order one half in time. More precisely $u_t$ is in the Bochner-Nikolskij space $\mathcal{N}^{\frac{1}{2},2}(a,T-a),L^2(\Omega))$; which means 
\[
\sup_{h<a}\int_{a}^{T-a}\int_\Omega \Bigabs{\frac{u_t(t+h)-u_t(t)}{h^\frac12}}^2 dx dt <\infty,
\]
for $a>0$ (see Theorem~\ref{thm:main}. 

Our second result concerns spatial derivatives of $u_t$. In the degenerate case the existence of the space gradient of $u_t$ is not necessarily available. Therefore the fractional regularity of $u_t$ in space direction is of interest. 

In case of the whole space  $\Omega=\setR^n$ (or in the space periodic case) we can prove that $u_t$ has fractional spatial derivatives. 

 For all $1<p<\infty$ we prove a global estimates, i.e. 
\[
\sup_{h<a}\int_{a}^{T-a}\int_{\setR^n} \Bigabs{\frac{u_t(t,x+he_i)-u_t(t,x)}{h^\frac14}}^2 dx dt <\infty,
\]
for $e_i$ the $i$-th unit vector. This implies that $u_t\in L^2((a,T-a),\mathcal{N}^{\frac{1}{4},2}(\setR^n))$ (see Theorem~\ref{thm:main2}). 

These estimates can be refined in case the solutions are space-periodic. In this case we can show that $u_t\in L^2_{\text{loc}}((0,T),\mathcal{N}^{\frac{1}{2},2}_{\text{loc}}(\setR^n))$, provided $2-\frac2n< p\leq4$, furthermore we obtain Nikolskij differentiability orders in the interval $(\frac14,\frac12]$ for all $2-\frac2n\leq p<\infty$ (see Theorem~\ref{cor:alz}).

It turned out to be more difficult to gain regularity in space direction then in time. If one considers bounded domains one gets non-trivial additional boundary terms; a possible way of treating also this situation looks sophisticated and we postpone it to a future project.

In the special case of the parabolic p-Laplace one will find quite satisfactory estimates of $u$ and $\nabla u$ in the literature. Particular, the gradient is H\"older continuous provided the right hand side is smooth~\cite[Chapter 8]{DiB93} and \cite{Mis02,Sch13}. This is the parabolic version of the elliptic analogue initiated by Uhralzeva~\cite{Ura68} for equations and Uhlenbeck~\cite{Uhl77} for systems.
Moreover, an $L^q$-Theory~\cite{AceMin07} (by non-linear \Calderon-Zygmund theory initiated by Iwaniec~\cite{Iwa83}) and pointwise estimates of Riesz type are available for the parabolic p-Laplace~\cite{KuuMin123,KuuMin13}. The only results for the time derivative that is available is see~\cite{Lin12} and \cite[Theorem~8.1]{Lio69}.

In the next two paragraphs we introduce two generalizations of the parabolic p-Laplace for which we can prove the very same regularity results.
\subsection{General systems}
We present the general case of systems where the elliptic operator is Euler of a variational problem.

We consider the following energy functional. For $v:\Omega\to \setR^N$
\begin{align}
 \label{energy1}
\mathcal{J}(v)=\int_{\Omega}F(x,Dv)-fv dx.
\end{align}
The respective Euler operator is defined as $-\Div (\partial_z F).$
Again we will use $Dv$ as a substitute for both the gradient as well as the symmetric gradient as we can treat both cases simultaneously. 

At this point we discuss the assumptions for systems with p-growth. Observe, that we will include later generalized conditions of Orlicz growth (see Assumptions~\ref{ass:phi} below). 
\begin{assumption}
\label{ass:p}
For $x\in\setR^n$ and $Q,P\in\setR^{nN}$ we assume the following growth conditions for $\mu\geq 0$
\begin{enumerate}
 \item  $F$ is measurable and $F(x,\cdot)\in C^{0,1}(\setR^{n\times N})\cap C^2(\setR^{n\times N}\setminus\set{0}).$
\item For $A_{i,j}(x,Q):=\partial_{Q_{ij}} F(x,Q)$, $\lambda (\abs{Q}^2+\mu^2)^\frac{p-2}{2}\abs{Q}^2\leq A(x,Q)\cdot Q$ and $\abs{A(x,Q)}\leq \Lambda (\abs{Q}^2+\mu^2)^\frac{p-2}{2}\abs{Q}$, for $0<\lambda\leq \Lambda$. 
\item $ (A(x,Q)-A(x,P))\cdot(Q-P)\sim \abs{V(Q)-V(P)}^2$, for $V(Q):=(\mu^2+\abs{Q}^2)^\frac{p-2}4Q$.
\item $\abs{D_x A(x,Q)}\leq (\abs{Q}^2+\mu^2)^\frac{p-2}{2}\abs{Q}$.
\item 
$D_{Q}A(x,Q)P\otimes P\sim \abs{P}^2(\mu^2+\abs{Q}^2)^\frac{p-2}{2}$
\item $F(x,0)=0$. 
\end{enumerate}
In case we consider the symmetric gradient, we need additionally that
\begin{enumerate}
 \item[{\rm (g)}] If $Q$ is symmetric, then $A(x,Q)$ is also symmetric.
\end{enumerate}
The constants used in this assumption are called the \textit{characteristics of $F$}.
\end{assumption}
\begin{remark}
Well known is the case of Uhlenbeck structure; $F(x,\abs{Du})$. In case $ F(x,Du)=\frac{1}{p}\abs{Du}^p$ \eqref{eq} becomes the p-Laplace or the symmetric p-Laplace.

 It is possible to replace (b) by the assumption

$\lambda (\abs{Q}^2+\mu^2)^\frac{p-2}{2}\abs{Q}^2-k(x)\leq A(x,Q)\cdot Q$ and $\abs{A(x,Q)}\leq \Lambda (\abs{Q}^2+\mu^2)^\frac{p-1}{2}+K(x)$, with $k,K\in L^1(\setR^n)\cap L^\infty(\setR^n)$.

It is also possible to replace (d) by the assumption
$D_x A(x,Q)= B_1(x,Q)+B(x)$, such that $\abs{B_1(x,Q)}\leq  c(\abs{Q}^2+\mu^2)^\frac{p-2}{2}\abs{Q}$ and $\nabla B\in L^{p'}(Q_T)$. More general one considers $B$ to be of lower order. I.e. of the form $B(x,Q)$ has lower then $p-1$ growth in $\abs{Q}$. Finally it is possible to replace (f) by $F(x,0)=g(x)\in L^1(Q_T)\cap L^\infty(Q_T)$.
We did not include this weaker assumptions, as the proof would not change significantly but would be more complicated to read. 
\end{remark}

%

From the above energy functional \eqref{energy1} we associate the following flow. 
\begin{align}
 \label{eq}
\begin{aligned}
 u_t-\divergence(A(x, Du))&= f\text{ on } Q_T\\
u&=0\text{ on }(0,T)\times\partial \Omega\\
u(0,x)&=u_0(x)\text{ on }\Omega.
\end{aligned}
\end{align}
In the general case, in particular problems where the symmetric gradient is involved one knows only the improved differentiability based on classical difference quotient techniques. Recently in the PhD thesis of Burczak, some local (in space) estimates of fractional time derivative of $u$ is shown~\cite[Lemma~4.9]{BurThe}. See also~\cite{Bur14} for an overview of regularity of symmetric $p$-Laplace. 
Our paper gives an additional progress in the theory, by showing partial differentiability of $u_t$ in time and space direction.
Clearly there are many paper on partial regularity or short-time regularity, which is not of our concern here.

\subsection{Non-Newtonian fluids}
\label{ssec:nnfl}
Our regularity theory implies results for certain non Newtonian fluids; in explicit for solutions of the following system
\begin{align}
 \label{stokes}
\begin{aligned}
 u_t-\divergence(A(x, \bfepsilon u))+\nabla\pi&= f\text{ on } Q_T\\
\divergence(u)&=0\\
u&=0\text{ on }(0,T)\times\partial \Omega\\
u(0,x)&=u_0(x)\text{ on }\Omega.
\end{aligned}
\end{align}
Here $u:\Omega\to \setR^n$ describes the velocity and $\pi:\omega\to\setR$ describes the pressure of a fluid moving through time in $\Omega$. We denote the symmetric gradient as $\bfepsilon v:=\frac12(\nabla v+(\nabla v)^T)$. The vector field $f:\Omega\to\setR^{n\times n}$ is the outer force (e.g. the gravity) and $u_0$ some given initial state.
The matrix $A(x,\bfepsilon u)$ is the (non-linear) stress tensor. Our model situation is, that $A$ is of power law type, i.e. $A(x,\bfepsilon u)=\nu (\mu^2 + \abs{\bfepsilon u})^\frac{p-2}2\bfepsilon u$, for $1<p<\infty$, $\nu>0$ and $\mu\geq 0$. However we can also treat Carreau type fluids: $A(x,\bfepsilon u)=\nu_\infty \bfepsilon u+ \nu (\mu^2 + \abs{\bfepsilon u})^\frac{p-2}2\bfepsilon u$, for $\nu_\infty\geq 0$. 

Observe, that if $p=2$, i.e. $ A(x,\bfepsilon u)=\nu \bfepsilon u$ the system becomes the classical (linear) Stokes equation. The non-linear dependence of $\abs{\bfepsilon u}$ is due to the fact, that in the case of non-Newtonian fluids the viscosity changes with the velocity (more explicit by the shear rate).
Power-law or Carreau type fluids are therefore widely used among engineers. For a more detailed discussion of the physical model we refer to \cite{MalR05,Raj92} and references therein. 

Please observe, that the weak formulation of \eqref{eq} and \eqref{stokes} are the same on divergence free testfunctions. Indeed, Assumption~\ref{ass:p} and implies that there exists a unique\footnote{Here $V^{1,p}_{0,\divergence}(\Omega):=\overline{\set{g\in C^\infty_0(\Omega):\divergence g=0} }^{\norm{\nabla \cdot}_{L^p(\Omega)}}$.} $u\in L^p((0,T),V^{1,p}_{0,\Div}(\Omega))\cap L^\infty((0,T),L^2(\Omega))$, such that
\begin{align}
 \label{weak}
\begin{aligned}
&-\int_0^T\int_\Omega u\partial_t\xi + A(x, D u)D\xi\dx\dt+\int_\Omega u(y,T)\xi(T,x)-u_0(x)\xi(0,x)\dx\\
&\qquad=\int_0^T\int_\Omega f\xi\dx\dt\text{ for }\xi\in \set{C^\infty_0(\Omega),\,\Div \xi=0},
\end{aligned}
\end{align}
 for every right hand side $f\in L^{p'}([0,T),(V^{1,p}_{0,\Div}(\Omega))^*)$ and every $u_0\in L^2(\Omega)$. 

Although our techniques are shaped for the reduced Stokes power law model \eqref{stokes} they are of relevance for the full Navier-Stokes power law model that includes the convective term
\[
u_t-\Div(\nu\abs{\bfepsilon u+\mu}^{p-2}\bfepsilon u)+[\nabla u] u +\nabla\pi= f.
\]
 It is clear, that for sufficiently large $p$ the convective term can be treated as right hand side. 

For instationary power law fluids little regularity is known. In fact, the results mainly relay on classical difference quotients. If they are applied in time direction they imply $L^\infty((0,T),L^2(\Omega))$ estimates for $u_t$, see~\cite{BulEttKap10} (and Proposition~\ref{pro:Vt}). In \cite{KapMalSta02} the authors used space and time difference quotient to deduce by embedding H\"older continuity for gradients, in two space dimensions. See also~\cite{Kap05} where some fractional derivatives of $u_t$ are shown in the non-degenerate case and in two space dimensions.
   
The results present here might very well be an important step to be able to transfer stationary theory to the instationary case. However, namely in case of three space dimension the additional regularity $u_t\in \mathcal{N}^{\frac{1}{2}}((0,T),L^2(\Omega))$ and $u_t\in L^2((0,T),\mathcal{N}^{\frac14,2}(\setR^n)$ is of independent value. 
\smallskip

Let us briefly discuss the techniques we will use and develop.
For the estimate of the fractional time derivative we construct backward-forward quotients of the type
\begin{align}
\label{eq:quot}
\frac{u_t(t,x)-\dashint_0^hu_t(t+s,x)ds}{h}.
\end{align}
This quotient technique was developed in~\cite{FreSpe12a,FreSpe12b} but used on different types of PDE. 

 For the spatial estimates we will use what we call ``Queer Quotients'', which are a mixed in time-space difference quotient of of type \eqref{eq:quot}. By doing so we can use the variational structure of the elliptic part of the system. This can be combined with the estimate for pure in time fractional derivatives already obtained to get fractional differentiability in space for $u_t$. Interestingly the estimates are mainly driven by the term $u_t$; although the ellipticity assumptions are needed to treat the term $A(\cdot,Du)$. 
  
%
Indeed, the positive terms in our estimates will always be derived from the term $u_t$; not like more common from both terms of the right hand side.
 \begin{remark}
 \label{rem:1}
 Assumption~\ref{ass:p}-(e),  which might at first seem the most restrictive assumption is actually a consequence of (c).
As $A(x,\cdot)$ is differentiable we deduce from (c), that
\[
\frac1{h^2}(A(x,Q+hP)-A(x,Q))\cdot hP\sim\Bigabs{\frac{V(Q+hP)-V(Q)}{h}}^2.
\]
By letting $h\to 0$ we get $D_Q A(x,Q)P\otimes P\sim\abs{D_QV(Q)}^2\abs{P}^2$. Now (e) follows by calculating $D_Q V(Q)$. 
\end{remark}

\section{Assumptions and Main results}
Our estimates allow to assume growth condition related to a convex Orlicz function.
\begin{assumption}
 \label{ass:phi1}
Let $\phi\in C^2((0,\infty),(0,\infty))\cap C^1([0,\infty),[0,\infty))$, be a convex, $\phi'(0)=0$ and $\phi'(t)>0$, for $t>0$ function,  that holds $\phi''(t)t^2\sim \phi(t)$, for $t>0$. 
\end{assumption}
This includes $\phi(t)=t^p$, for $1<p<\infty$ and also $\phi(t)=\max\set{t^p,t^q}$, for two different exponents. Excluded are $t\log(t)$ on the one end and $e^t$ on the other end.
\begin{assumption}
\label{ass:phi}
Let $\phi$ hold Assumption~\ref{ass:phi1}. Additional we assume the following growth conditions
for $x\in\setR^n$ and $Q,P\in\setR^{nN}$ and $\mu\geq0$.
\begin{enumerate}
 \item  $F$ is measurable and $F(x,\cdot)\in C^{0,1}(\setR^{n\times N})\cap C^2(\setR^{n\times N}\setminus\set{0}).$
\item For $A(x,Q):=D_Q F(x,Q)$, $\lambda \phi'(\abs{Q}+\mu)\abs{Q}\leq A(x,Q)\cdot Q$ and $\abs{A(x,Q)}\leq \Lambda\phi''(\abs{Q}+\mu)\abs{Q}$. 
\item $ (A(x,Q)-A(x,P))(Q-P)\sim \abs{V(Q)-V(P)}^2$, for $V(Q):=\sqrt{\phi'(\abs{Q}+\mu)\abs{Q}}\frac{Q}{\abs{Q}}$.
\item $\abs{D_x A(x,Q)}\leq c\phi'(\abs{Q}+\mu)$.
\item $D_QA(x,Q)P\otimes P\sim \abs{P}^2\phi''(\abs{Q}+\mu)$.
\item $F(x,0)=0$.
\end{enumerate}
In case we consider the symmetric gradient, we need additionally that
\begin{enumerate}
 \item[{\rm (g)}] If $Q$ is symmetric, then $A(x,Q)$ is also symmetric.
\end{enumerate}
The constants used in this assumption are called the \textit{characteristics of $F$}.
\end{assumption}
In case of $\phi(t)=t^p$ we have the typical p-growth, which is included in the assumption above.

Assumption~\ref{ass:phi} include the following case of Uhlenbeck structure, that $F(x,Q)=\phi(\abs{Q})$. Then we have $A_{i,j}(x,Q):=\phi'(\abs{Q})\frac{Q_{ij}}{\abs{Q}}$. 
This particular case of Uhlenbeck structure is the subject of study by an increasing number of mathematicians nowadays. Many physical relevant models are of Orlicz growth and not $p$-growth. For example Carreau type fluids.

We introduce $\phi^*(t):=\sup_{a>0}(at-\phi(a))$ is the conjugate Orlicz function. The following inequality of Young type holds (by definition)
\begin{align}
\label{eq:young}
 ab\leq \phi^*(a)+\phi(b),\text{ for all }a,b\in [0,\infty).
\end{align}
The assumption $\phi''(t)t^2\sim \phi(t)$ implies that its conjugate $\phi^*$ has analogous properties, see~\cite{DieE08,DieKapSch11}.
The space $L^\phi(\Omega):=\set{f:\int_\Omega \phi(\abs{f})<\infty}$ is a Banach space endowed with the Luxembourg norm. Moreover, we define
\[
 V^{1,\phi}_0(\Omega)=\overline{C^\infty_0(\Omega)}^{\norm{\nabla \cdot}_{L^\phi(\Omega)}}\text{ and }V^{1,\phi}_{0,\Div}(\Omega):=\overline{\set{g\in C^\infty_0(\Omega):\Div g=0} }^{\norm{\nabla \cdot}_{L^\phi(\Omega)}}.
\]
For the definition and properties of $L^\phi(\Omega),W^{1,\phi}(\Omega)$ and for the analysis on divergence free Orlicz spaces we refer to~\cite{DieKap12,DieKapSch13}. The following estimate is important for us. It can be found in \cite[Lemma~2.4]{DieKapSch11},
\begin{align}
\label{eq:hammer} \Big(\frac{\phi'(\abs{Q}+\mu)Q}{\mu+\abs{Q}}-\frac{\phi'(\abs{P}+\mu)P}{\mu+\abs{P}}\Big)(Q-P)\sim\abs{V(Q)-V(P)}^2\sim \phi''(\mu+\abs{Q}+\abs{Q-P})\abs{P-Q}^2.
\end{align}
Together with Remark~\ref{rem:1}, this implies that whenever it is well defined
\begin{align}
\label{eq:hammer2}\abs{\partial_{i}V(Du)}^2\sim \phi''(\mu+\abs{Q})\abs{\partial_{i}Du}^2.
\end{align}
Now, standard existence theory implies that if Assumption~\ref{ass:phi} holds, then there exists a unique solution of \eqref{eq} or \eqref{stokes} for $u_0\in L^2(\Omega)$ and $f\in L^{\phi^*}([0,T),(V^{1,\phi}_0)^*)$. Nevertheless, because we need more regularity of the solutions and to keep the paper self contained, we include an existence result (Proposition~\ref{pro:Vt}).

%
%
We can now state our results. The first theorem concerns fractional derivatives of $u_t$ in time direction.
\begin{theorem}
 \label{thm:main}
Let $F$ hold Assumption~\ref{ass:p} or Assumption~\ref{ass:phi}. Let $u_0\in L^2(\Omega)$ and $f\in L^{\phi^*}([0,T),(V^{1,\phi}_0(\Omega))^*)$.
Let $u$ be a solution of \eqref{eq} or \eqref{stokes}. If additionally $f\in W^{1,2}((0,T),L^2(\Omega))$, then $u_t\in \mathcal{N}^{\frac12,2}((a,T-a),L^2(\Omega))$ and 
\begin{align*}
 \norm{u_t}_{\mathcal{N}^{\frac12,2}((a,T-a),L^2(\Omega))}\leq \frac{K}{a^2}
\end{align*}
for every $a>0$.
The constant $K$ only depends on the characteristics of $F$ and the regularity of $f$.
\end{theorem}
In case of a bounded domain $\Omega$, the following estimate is available
\begin{align*}
 \norm{u_t}_{\mathcal{N}^{\frac12,2}((a,T-a),L^2(\Omega))}\leq \frac{c}{a^2}\Big(\norm{\phi^*(\abs{f})}_{L^{1}(Q_T)}+\norm{f}_ {W^{1,2}((0,T),L^2(\Omega))}\Big).
\end{align*}
The second theorem concerns fractional derivatives of $u_t$ in space direction.
\begin{theorem}
 \label{thm:main2}
Let $F$ hold Assumption~\ref{ass:p} or Assumption~\ref{ass:phi} on the wholespace $\setR^n$. Let $u_0\in L^2(\setR^n)$ and $f\in L^{\phi^*}([0,T),(V^{1,\phi}_0(\setR^n))^*)$.
Let $u$ be a solution of \eqref{eq} or \eqref{stokes}. 

If additionally $f\in \mathcal{N}^{\frac12,2}(Q_T)$, $f_t\in L^2(Q_T)$ and $\nabla f\in L^{\phi^*}(Q_T)$, then, 
\begin{align*}
 \norm{u_t}_{L^2((a,T-a),\mathcal{N}^{\frac14,2}(\setR^n)}\leq \frac{K}{a^2}. 
\end{align*}
The constant $K$ only depends on the characteristics of $F$ and the regularity of $f$. 
\end{theorem}
Theorem~\ref{thm:main} and Theorem~\ref{thm:main2} and \eqref{eq:se} imply the following corollary
\begin{corollary}
 With the same hypothesis as in Theorem~\ref{thm:main2}, we find for $u$ a solution of \eqref{eq} or \eqref{stokes}, that $u_t\in \mathcal{N}^{\frac14,2}_{\local}(Q_T)$ and therefore
$u_t\in L^q_{\local}(Q_T)$, for $q\in [2,2\frac{n+1}{n+\frac12})$, by \eqref{eq:se}.
\end{corollary}

In case $F$ has p-growth (i.e. Assumption~\ref{ass:p} holds) we can refine the regularity of $u_t$ in dependence of $p$. If $p\geq 2$ we have to restrict to space periodic solutions. I.e. solutions in
\[
V^{1,p}_{\text{per}}(\setR^n)=\overline{\Bigset{g\in C^\infty_0(\setR^n):g(a+x)=f(x)\text{ for some }a\in\setR^n\text{ and }\int_{\setR^n} g=0}}^{\norm{\nabla\cdot}_p}.
\]
This is a closed subspace of $V^{1,p}(\setR^n)$, therefore a solution of \eqref{eq} exists. The solution then is periodic and it is enough to estimate it over one period, which is a bounded cube. These solutions are artificial, however, they are useful to introduce new techniques and new observations. In our case we introduce periodic solutions to empathize possible extensions of regularity for $u_t$ which are to some extend inherited in our techniques.   

\begin{theorem}
 \label{cor:alz}
With the same hypothesis as in Theorem~\ref{thm:main2} but replacing Assumption~\ref{ass:phi} by Assumption~\ref{ass:p}. In case the system \eqref{eq} is in terms of the full gradient we have:
\begin{enumerate} 
\item $p\geq 2$ and $u\in V^{1,p}_{\text{per}}(\setR^n)$, 

 $u_t\in L^2_{\text{loc}}((0,T),\mathcal{N}^{\beta,2}_{\text{loc}}(\setR^n))$. For $\beta=\min\set{\frac{1}{2},\frac14+\frac{1}{p}}$

Moreover, $u_t\in L^2_{\text{loc}}((0,T),L^q_{\text{loc}}(\setR^n)$ for $q\in \big[1,\min\bigset{\frac{2n}{n-\frac{p+4}{2p}},\frac{2n}{n-1}}\big)$, by embedding.  
\item $1<p\leq 2$ and $u\in V^{1,p}(\setR^n)$
$u_t\in L^2((a,T-a),\mathcal{N}^{1,\beta}(\setR^n))$, for 
\[ \beta=\min\Bigset{\frac12,\max\bigset{\frac{p+2-(2-p)\frac{n}2}{2p},\frac34-\frac{n}8(2-p),\frac14}}.
\]
\end{enumerate}
\end{theorem}
For \eqref{stokes} this improvement does not work. However, it is possible to show $u_t$ has a $\beta$-fractional symmetric gradient. See Remark~\ref{rem:fracsym}.

We conclude the section with the following corollary. It collects the differentiability regularity on $u_t$ and Sobolev embedding \eqref{eq:se}. 
\begin{corollary}
With the same hypothesis as in Theorem~\ref{cor:alz} we have
  $u_t\in L^\infty_\local((0,T),L^2(\setR^n))\cap L^a_\local(Q_T)\cap L^2_\local((0,T),L^q_\local(\Omega)$.
If 
\begin{enumerate}
 \item[(2D)] $q\in[1,4), a\in[1,3)$ for $p\in(1,4]$ or $q\in [1,\frac{8}{3}], a\in [1,\frac{12}{5}]$ for $p\in (1,\infty)$, 
  \item[(3D)] $q\in [1,3), a\in [1,\frac{8}{3})$ for $p\in[\frac{4}{3},4]$ or $q\in [1,\frac{12}{5}], a\in [1,\frac{16}{7}]$ for $p\in (1,\infty)$,
\item[(nD)] $q\in [1,\frac{2n}{n-1}), a\in [1,2\frac{2(n+1)}{n})$ for $p\in [2-\frac{2}{n},4]$ or $q\in [1,\frac{2n}{n-\frac12}], a\in [1,\frac{n+1}{n+\frac12}]$ for $p\in (2-\frac{4}{n},\infty)$.
\end{enumerate}
\end{corollary}

\section{Preliminary}
\subsection{Notation}
We will write $g\sim h$ if there exist two constants $c,C$, such that $cg\leq h\leq Cg$.
For a set $A$ with finite measure, we use
\[
 \dashint_A g dx:=\frac{1}{\abs{A}}\int_A g dx.
\]

We define for $g:\setR^n\to\setR^N$
\[
 \Delta^h_{x_i}(g):=(g(x_1,..,x_i+h,...,x_n)-g(x_1,..,x_i,...,x_n))
\] 
and
\[
D^h_{x_i}(g):=\frac1h(g(x_1,..,x_i+h,...,x_n)-g(x_1,..,x_i,...,x_n)).
\]
For an arbitrary direction $v\in\mathcal{S}^{n-1}$ we define
\[
 \Delta^h_{v}(g):=(g(x+hv)-g(x))
\] 
and
\[
D^h_{v}(g):=\frac1h(g(x+hv)-g(x)).
\]
 
Moreover, we define the "Queer Quotient" as
\[
 \Delta_{\nearrow}^h(g):=(g(x_1,..,x_i+h,...,x_j+h,...,x_n)-g(x_1,..,x_i,...,x_j+h,...,x_n)),
\]
and
\[
 D_{\nearrow}^h(g):=\frac1h(g(x_1,..,x_i+h,...,x_j+h,...,x_n)-g(x_1,..,x_i,...,x_j+h,...,x_n)),
\]
for arbitrary $i,j$.

Remark, that in case of $u$ being a solution of \eqref{stokes}, difference-quotients of $u$ are divergence-free and therefore  suitable test functions for \eqref{weak}.

We will use the following partial summation rules, without further reference
\[
 \int_a^b\Delta^h(f)(x)g(x)\dx=\int_{a+h}^bf(x)\Delta^{-h}(g)(x)\dx+\int_b^{b+h}f(x)g(x-h)-\int_a^{a+h}f(x)g(x),
\]
\[
 \int_a^b\Delta^{-h}(f)(x)g(x)\dx=\int_{a}^{b-h}f(x)\Delta^{h}(g)(x)\dx-\int_{b-h}^{b}f(x)g(x)+\int_{a-h}^{a}f(x+h)g(x)
\]
and
\[
 \int_a^b D^h(f)(x)g(x)\dx=-\int_{a+h}^bf(x)D^{-h}(g)(x)\dx+\dint_b^{b+h}f(x)g(x-h)-\dint_a^{a+h}f(x)g(x).
\]
Please notice also the following cancellation property we will use
\begin{align}
\label{basic:1}
 \int_a^bf(t+h)-f(t-h)\dt=\int_{a+h}^{b+h}f(t)\dt-\int_{a-h}^{b-h}f(t)\dt=\int_{b-h}^{b+h}f(t)\dt-\int_{a-h}^{a+h}f(t)\dt.
\end{align}
Finally the following observation is needed (for the sake of completeness).
We find for $g\in W^{1,1}((a,b))$ and for  $s<h$ 
\begin{align}
\label{basic:2}
 \int_a^{b-h}\abs{(D^sg)(t)}dt\leq \int_a^{b-h}\dashint_0^s\abs{g'(t+s)}dsdt=\dashint_0^s\int_{a+s}^{b-h+s}\abs{g'(t)}dt ds\leq \int_a^b\abs{g'(t)} dt.
\end{align}

\subsection{Preliminary estimates}
We start with the following elementary pointwise estimate for real functions. It is, however, essential for all our estimates.
We use the following estimate which is a direct consequence of \cite[Lemma 2.7]{DieSV09}. 
If $\phi''(t)t^2\sim \phi(t)$ and $a_0,a_1\in\setR^n$ we have 
\begin{align}
 \label{eq:equiv}
\phi''(\abs{a_1}+\abs{a_0})\sim\int_0^1 \phi''(\abs{\theta a_0+(1-\theta)a_1})d\theta.
\end{align}
Observe, that although the estimates in Lemma~\ref{lem:t} and Corollary~\ref{cor:schraeg} below are purely analytic (i.e. not depending on the PDE). We use the notation of the PDE as we will use it.
\begin{lemma}
\label{lem:t}
 Let $F$ hold Assumption~\ref{ass:p} or \ref{ass:phi}. And $Du:\Omega \to\RNn$. There exists constants depending only on the characteristics of $F$, such that for every $h>0$ the following identities hold. 
\begin{enumerate}
 \item $A(Du)(t)\cdot D^h_{t}(Du)(t)+\frac{c_1}{h}\abs{\Delta^{h}_{t}V(Du)(t)}^2\leq D^h_{t}F(Du)(t)\leq A(Du)(t)\cdot D^h_{t}(Du)(t)+\frac{c_2}{h}\abs{\Delta^{h}_{t}V(Du)(t)}^2$,
\item $A(Du)(t)\cdot D^{-h}_{t}(Du)(t)-\frac{c}{h} \abs{\Delta^{-h}_{t}V(Du)(t)}^2\leq D^{-h}_{t}F(Du)(t)\leq A(Du)(t)\cdot D^{-h}_{t}(Du)(t)+\frac{c}{h} \abs{\Delta^{-h}_{t}V(Du)(t)}^2$,
\item \begin{align*}&A(Du)(t)\cdot \frac{1}{2h}\big(Du(t+h)-Du(t-h)\big)+\frac{c_1}{2h}\abs{\Delta^{h}_{t}V(Du)(t)}^2-\frac{c}{2h} \abs{\Delta^{-h}_{t}V(Du)(t)}^2\\
&\leq \frac{1}{2h} \big(F(Du(t+h))-F(Du(t-h))\big)\\
&\leq
A(Du)(t)\cdot \frac{1}{2h}\big(Du(t+h)-Du(t-h)\big)+\frac{c_2}{2h}\abs{\Delta^{h}_{t}V(Du)(t)}^2+\frac{c}{2h} \abs{\Delta^{-h}_{t}V(Du)(t)}^2\end{align*}
\end{enumerate}
for any fixed value. I.e. $A(\cdot):=A(x,\cdot)$ and $F(\cdot):=F(x,\cdot)$.
\end{lemma}
\begin{proof}
We use the monotonicity of Assumption~\ref{ass:phi}, to find
\begin{align*}
& D^h_{t}F(\cdot, D u)(t)=\frac1h\int_0^1\frac{d}{ds}F(sDu(t+h)+(1-s)Du(t))\ds\\
&=\int_0^1A(sDu(t+h)+(1-s)Du(t))\cdot D^h_t(Du)(t)\ds\\
&=\int_0^1(A(sDu(t+h)+(1-s)Du(t))-A(Du(t)))\cdot D^h_t(Du)(t)\ds\\
&\quad +\int_0^1A(Du(t))\cdot D^h_t(Du)(t)\ds\\
&=\int_0^1\frac{1}{sh}(A(Du(t)+s(Du(t+h)-Du(t)))-A(Du(t)))\cdot s(Du(t+h)-Du(t))\ds\\
&\quad +A(Du(t))\cdot D^h_t(Du)(t)
\end{align*}
We estimate using Assumption~\ref{ass:phi}(d) and \eqref{eq:hammer} to gain
\begin{align*}
 \int_0^1&\frac{1}{sh}\big(A(Du(t)+s\Delta^h_tDu(t))-A(Du(t))\big)\cdot s(\Delta^h_t(Du)(t))\ds\\
&\sim \int_0^1\frac{1}{sh}\phi''(\mu+\abs{s\Delta^{h}_t(Du)(t)}+\abs{Du(t)})s^2\abs{\Delta_t^{h}Du}^2\ds.
\end{align*}
Because $\abs{s\Delta^{h}_t(Du)(t)}+\abs{Du(t)}= s(\abs{\Delta^{h}_t(Du)(t)}+\abs{Du(t)})+(1-s)\abs{Du(t)}$ we use \eqref{eq:equiv}, Assumption~\ref{ass:phi} and \eqref{eq:hammer} to find
\begin{align*} \int_0^1&\frac{1}{sh}\phi''(\mu+\abs{s\Delta^{h}_t(Du)(t)}+\abs{Du(t)})s^2\abs{\Delta_t^{h}Du}^2\ds\sim \frac1{h}\abs{\Delta^{h}_{t}V(Du)(t)}^2.
\end{align*} 
This implies the first estimate.

We redo the argument to find that
\begin{align*}
D^{-h}_t&F(Du)(t)= D^h_{t}F(D u)(t-h)\\
&=\frac1h\int_0^1\frac{d}{ds}F(sDu(t)+(1-s)Du(t-h))\ds\\
&=\int_0^1A(sDu(t)+(1-s)Du(t-h))\cdot D^h_t(Du)(t-h)\ds\\
&=\int_0^1\big(A(Du(t)+(1-s)(-\Delta^hDu(t-h)))-A(Du(t))\big)\cdot D^h_t(Du)(t-h)\ds\\
&\quad +\int_0^1A(Du(t))\cdot D^h_t(Du)(t)\ds\\
&= -\int_0^1\frac{1}{(1-s)h}\big(A(Du(t)-(1-s)\Delta^h_tDu(t-h))-A(Du(t))\big)\cdot (1-s)(-\Delta^h_t(Du)(t-h))\ds\\
&\quad +A(Du(t))\cdot D^h_t(Du)(t-h)
\end{align*}
 We estimate again using Assumption~\ref{ass:phi}(d) and \eqref{eq:hammer}
\begin{align*} \int_0^1&\frac{1}{(1-s)h}\big(A(Du(t)+(1-s)\Delta^h_tDu(t-h))-A(Du(t))\big)\cdot (1-s)(-\Delta^h_t(Du)(t))\ds\\
&\sim \int_0^1\frac{1}{(1-s)h}\phi''(\mu+\abs{(1-s)\Delta^{-h}_t(Du)(t)}+\abs{Du(t)})(1-s)^2\abs{\Delta_t^{-h}Du}^2\ds
\end{align*}
Because $\abs{s\Delta^{-h}_t(Du)(t)}+\abs{Du(t)}= s(\abs{\Delta^{-h}_t(Du)(t)}+\abs{Du(t)})+(1-s)\abs{Du(t)}$ we can use \eqref{eq:equiv}, \eqref{eq:hammer} and Assumption~\ref{ass:phi} to find
\begin{align*} \int_0^1&\frac{1}{(1-s)h}\phi''(\mu+\abs{(1-s)\Delta^{-h}_t(Du)(t)}+\abs{Du(t)})(1-s)^2\abs{\Delta_t^{-h}Du}^2\ds\\ 
&\sim \frac1{h}\abs{\Delta^{-h}_{t}V(Du)(t)}^2.
\end{align*} 
This implies the second estimate.

For the third estimate we use
\begin{align*}
 &A(Du(t))\cdot (Du(t+h)-Du(t-h))\\
&\quad=A(Du(t))(Du(t+h)-Du(t))-A(Du(t))\cdot(Du(t-h)-Du(t)),
\end{align*}
therefore 
\[
 A(Du(t))\cdot \frac{1}{2h}\big(Du(t+h)-Du(t-h)\big)=\frac12\big(A(Du(t))D^h_t(Du(t))+A(Du(t))D^{-h}_t(Du(t))\big),
\]
 which implies by estimate one and two the third estimate because
\[
  D^h_{t}F(Du)(t)+ D^{-h}_{t}F(Du)(t)= \frac1{h} \Big(F(Du(t+h)-F(Du(t-h))\Big).
\]

\end{proof}
\begin{corollary}
\label{cor:schraeg}
For every $h>0$ and for $v\in\mathcal{S}^{n}$ the following identities hold. 
\begin{enumerate}
 \item $A(Du)(z)\cdot D^h_{v}(Du)(z)+\frac{c_1}h\abs{\Delta^{h}_{v}V(Du)(z)}\leq D^h_{v}F(Du)(z) \leq A(Du)(z)\cdot D^h_{v}(Du)(z)+\frac{c_1}h\abs{\Delta^{h}_{v}V(Du)(z)}$,
\item $A(Du)(z)\cdot D^{-h}_{v}(Du)(z)-\frac{c}h\abs{\Delta^{h}_{v}V(Du)(z)}\leq D^{-h}_{v}F(Du)(z) \leq A(Du)(z)\cdot D^{-h}_{v}(Du)(z)+\frac{c}h\abs{\Delta^{h}_{v}V(Du)(z)}$,
\item \begin{align*}
       &A(Du)(z)\cdot \frac{1}{2h}\big(Du(z+hv)-Du(x-hv)\big)+\frac{c_1}{2h}\abs{\Delta^{-h}_{v}V(Du)(z)}-\frac{c}{2h}\abs{\Delta^{h}_{v}V(Du)(z)}\\
&\leq \frac{1}{2h} \big(F(Du(z+hv))-F(Du(z-hv))\big)\\
&\leq  A(Du)(z)\cdot \frac{1}{2h}\big(Du(z+hv)-Du(x-hv)\big)+\frac{c_2}{2h}\abs{\Delta^{-h}_{v}V(Du)(z)}+\frac{c}{2h}\abs{\Delta^{h}_{v}V(Du)(z)}.
      \end{align*}
\end{enumerate}
The constants only depend on the characteristics of $F$.
\end{corollary}
\begin{proof}
 As the arguments in the proof of Lemma~\ref{lem:t} works for arbitrary directions the proof is line by line the same.
\end{proof}
  We will need the following analytic Lemma, which was hinted to us by L. Diening. A similar estimate can be found in \cite[Lemma 12]{DieE08}. It is an integral characterization of Nikolskij spaces. 
\begin{lemma}
\label{lem:diening}
 Let $1\leq q\leq\infty$. If $f\in L^q([a,b+H])$. If 
\begin{align}
\label{eq:die}
 \sup_{H\geq h>0}\Bignorm{\dint_{\nsk 0}^h\abs{f(t+s)-f(t)}\ds}_{L^q([a,b+h])} \leq h^\alpha K, 
\end{align}
then $f\in \mathcal{N}^{\alpha,q}([a,b])$, moreover
\[
 \sup_{\frac{H}{2}\geq h>0}\norm{\abs{f(t+h)-f(t)}}_{L^q([a,b])}\leq 3h^\alpha K. 
\]
\end{lemma}
\begin{remark}
 Clearly, the reverse direction is true as well. Simply because
\[
 \dint_{\nsk 0}^h\abs{f(t+s)-f(t)}\ds\leq \sup_{h\geq s>0} \abs{f(t+s)-f(t)}.
\]
Previously one used the slightly weaker but more obvious statement that the bound of \eqref{eq:die} implied that $f\in \mathcal{N}^{\beta,q}((a,b))$, for all $\beta<\alpha$~\cite[Appendix]{FreSpe12b}.
\end{remark}
\begin{proof}[Proof of Lemma~\ref{lem:diening}]
 We estimate
\begin{align*}
 \abs{f(t+h)-f(t)}&= \Bigabs{\dint_{\nsk 0}^hf(t+h)- f(t+h+s)\ds+\dint_{\nsk 0}^hf(t+h+s)-f(t)\ds}\\
&\leq \dint_{\nsk 0}^h\abs{f(t+h)-f(t+h+s)}\ds +\dint_{\nsk 0}^h\abs{f(t)-f(t+h+s)}\ds\\
&\leq \dint_{\nsk 0}^h\abs{f(t+h)-f(t+h+s)}\ds +2 \dint_{\nsk 0}^{2h}\abs{f(t)-f(t+a)}\da.
\end{align*}
This implies, for $2h\leq H$
\begin{align*}
& \norm{\abs{f(\cdot+h)-f(\cdot)}}_{L^q([a,b])}\\
&\quad\leq \Bignorm{{\dint}_{\nsk 0}^h\abs{f(\cdot+s)-f(\cdot)}\ds}_{L^q([a+h,b+h])} + 2 \Bignorm{\dint_{\nsk 0}^{2h}\abs{f(\cdot+s)-f(\cdot)}\ds}_{L^q([a,b])} \leq 3h^\alpha K.
\end{align*}

\end{proof}

\section{Estimates in time direction}
\label{sec:time}
We introduce some natural estimates, which are closely connected to the existence theory for \eqref{eq} and \eqref{stokes}. 


\begin{proposition}
\label{pro:Vt}
 Let Assumption~\ref{ass:phi} hold, $0<2a<T$ and $f\in L^2(Q_T)\cap L^{\phi^*}((0,T),(V^{1,\phi}(\Omega))^*)$. Then there exists a unique solution $u$ of \eqref{eq} or \eqref{stokes}, such that
\begin{align*}
 \int_a^T\int_{\Omega}\abs{u_t}^2\dt\dx+ \sup_{(a,T)}\int_{\Omega}F(Du)\dx \leq  \frac{c}{a}\int_0^T\int_{\Omega} \phi(\abs{Du})+\abs{f}^2 \dx\dt.
\end{align*}
If additionally $f\in W^{1,2}((0.T);L^{2}(\Omega))$ 
\[
 \int_{2a}^T\int_{\Omega}\abs{\partial_t V(Du)}^2\dx\dt+\sup_{t\in[2a,T]}\int_{\Omega}\abs{u_t}^2\dx\leq c \int_a^T\int_{\Omega}\abs{\partial_t f}^{2}\dx\dt+\frac{c}{a}\int_a^T\int_{\Omega}\abs{u_t}^2\dx.
\]
\end{proposition}
The above estimates are standard. Formally they can be derived by using $u_t$ and $\partial_t^2 u$ as testfunctions. These are solenoidal, such that the pressure does not interfere. In case of p-structure the second estimate can be found in \cite[Theorem~8.1]{Lio69}. For the Orlicz setting we refer to~\cite{BulEttKap10}. As our assumptions are not immediately covered by these references. For this reason and to keep the paper self-contained we include a proof of Proposition~\ref{pro:Vt} in the appendix below. 

To simplify notation we will assume, that $u$ holds \eqref{eq} or \eqref{stokes} on a larger interval say $(-A,T+A)$ with initial data at the point $-A$. We can then assume in the following (by using \eqref{basic:2}), that for all $h\in (0,A]$
\begin{align}
 \label{eq:timeall}
\int_0^T\int_\Omega&\abs{D^h_t V(Du)}^2+\abs{u_t}^2\dx\dt +\sup_{t\in[0,T]}\int_{{\Omega}}\abs{u_t}^2+ F(Du)\dx\leq \frac{K}{A},
\end{align}
Here $K$ depends on $f$ by the right hand side of Proposition~\ref{pro:Vt}. If $\Omega$ is bounded one can estimate $\int_0^T\int_\Omega \phi(\abs{Du})\dx\dt\leq \int_0^T\int_\Omega \phi^*(\abs{f})$, by testing with $u$, \Poincare{} and Young' s inequality for Orlicz functions \eqref{eq:young}. Then 
\[
K\sim \int_0^T\int_{\Omega} \phi^*(\abs{f})+\abs{f}^2+\abs{f_t}^2 \dx\dt.
\]
\begin{remark}
\label{rem:pleq2}
 Please observe, that in case of p-growth and $1<p\leq 2$, this implies that $Du_t\in L^p((a,T)\times \Omega)$. As Young's inequality and \eqref{eq:hammer2} imply
\[ \abs{\partial_tDu}^p=(\abs{Du}^{p-2}\abs{\partial_tDu}^2)^\frac{p}{2}\abs{Du}^\frac{(2-p)p}{2}\leq \abs{Du}^{p-2}\abs{\partial_tDu}^2+\abs{Du}^p\sim \abs{\partial_t V(Du)}+\abs{Du}^p.
\]
\end{remark}
We will need to use $u_t$ as a testfunction for our general assumptions. The next lemma justifies it.
\begin{lemma}
\label{lem:adm}
Let $u$ be a solution to \eqref{eq} or \eqref{stokes}, such that \eqref{eq:timeall} holds. Then $u_t$ can be used as a testfunction to \eqref{eq} or \eqref{stokes}, in the sense that
\begin{align*}
  \int_a^{T-a}\int_\Omega A(\cdot,Du)\cdot Du_t + u_t\cdot u_t\dx\dt = \int_a^{T-a}\int_\Omega f\cdot u_t.
  \end{align*}
\end{lemma}
\begin{proof}
We start with the following observation which follows from \eqref{eq:hammer2} and Assumption~\ref{ass:phi} 
\begin{align*}
\int_a^T\int_\Omega \abs{A(\cdot,Du)}\abs{Du_t}\leq &\leq c\int_a^T\int_\Omega \phi''(\mu+\abs{Du})\abs{Du}\abs{Du_t}\dz \\
&\leq c\int_a^T\int_\Omega \phi''(\mu+\abs{Du})(\abs{Du}^2+\abs{Du_t}^2)\dz\\
&\leq c\int_a^T\int_\Omega \phi(\abs{Du})+\abs{\partial_t V(Du)}^2\dz.
\end{align*}
This implies that $A(\cdot,Du)\cdot Du_t\in L^1((a,T)\times\Omega)$. Next we will approximate $u_t$. 
We use the approximation $D^h_tu$, for $0<h<a$ which is a testfunction to the system on the interval $[a,T-a]$. We then have pointwisely, that
\[
A(x,Du(t,x))\cdot D^h_t(Du)(t,x)\to A(x,Du(t,x))\cdot Du_t(t,x)\,\text{for } h\to 0
\]
and almost every $(t,x)$. To be able to use Lebesgues convergence theorem we have to provide a majorant. Let $t\in[a,T-a]$. Please observe, that by \eqref{eq:hammer} we have that
\[
\abs{D^h_tV(Du)(t)}^2\sim{\phi''(\mu+\max\set{\abs{Du(t)},\abs{Du(t+h)}})}\abs{D^h_t(Du)(t)}^2.
\]

If $\abs{Du(t+h)}\leq \abs{Du(t)}$, we estimate using the previous and Young's inequality
\[
\abs{A(Du(t))\cdot D^h_t(Du)}\leq {\phi''(\mu+\abs{Du(t)})}\abs{Du}\abs{D^h_t(Du)(t)}\leq c\phi(\abs{Du})+c\abs{D^h_tV(Du)(t)}^2.
\]
In case $\abs{Du(t+h)}\geq \abs{Du(t)}$, we estimate using Assumption~\ref{ass:phi} and the last estimate to get 
\begin{align*}
\abs{A(Du(t))\cdot D^h_t(Du)}&\leq \abs{(A(Du(t))-A(Du(t+h)))\cdot D^h_t(Du)}+\abs{A(Du(t+h))\cdot D^h_t(Du)}\\
&\leq c\abs{D^h_tV(Du)(t)}^2+c\abs{D^h_tV(Du)(t+h)}^2+c\phi(\abs{Du(t+h)}).
\end{align*}
All together we have the majorant $c(\abs{D^h_tV(Du)(t)}^2+\abs{D^h_tV(Du)(t+h)}^2+\phi(\abs{Du(t+h)})+\phi(\abs{Du(t)}))$ which is converging in $L^1((a,T-a)\times\Omega)$.
As we assume $f\in L^2(Q_T)$ and as we know by \eqref{eq:timeall} that $u_t\in L^2(Q_T)$, we find that
  \begin{align*}
  \int_a^{T-a}\int_\Omega A(\cdot,Du)\cdot DD^h_tu + u_t\cdot D^h_tu\dx\dt = \int_a^{T-a}\int_\Omega f\cdot D^h_tu,
  \end{align*} 
  converges to the right limit equation with $h\to0$. 
\end{proof}
 
The following Proposition is the main effort to prove Theorem~\ref{thm:main}.
\begin{proposition}
\label{pro:N12t} Let the assumptions of Proposition~\ref{pro:Vt} be satisfied, such that \eqref{eq:timeall} holds.
Then $u_t\in \mathcal{N}^{\frac12,2}((0,T),L^2(\Omega))$. Moreover, for every $A\geq h>0$ we have 
\begin{align*}
 & \dint_{\nsk 0}^h\int_{A+h}^{T-A-h}\int_{\Omega}\frac{\abs{\Delta^s_t(u_t)(t)}^2}{h}\dx\dt\ds\, da
\leq \frac{K}{A},
\end{align*}
where $K$ depends (linearly) on the right hand side of \eqref{eq:timeall} and on the characteristics of $F$.
\end{proposition}

\begin{proof}
Let $A\geq h>0$ and $a\in[0,h]$.
 We use the testfunction $\frac1h\dint_{ 0}^h\Delta^{-s}_t\Delta^s_t(u_t)(t)\ds$ on $[a+h,T-a-h]$. Let us briefly show we are allowed to use this test function for every positive $h$. We calculate
\[
 \dint_{ 0}^h\Delta^{-s}_t\Delta^s_t(u_t)(t)\ds= 2u_t-u(t+h)-u(t-h),
\]
the functions $u(t+h),u(t-h)\in V^{1,\phi}(Q_T)$ and therefore admissible. The function $u_t$ is admissible by Lemma~\ref{lem:adm}. 
%

Therefore we may apply the testfunction to our system. Partial summation implies 
\begin{align*}
 (I)+(II)+(III)&:=\frac1h\dint_{\nsk 0}^h \int_{\Omega}\int_{a+h}^{T-a-h}\unite\abs{\Delta^s_t(u_t)(t)}^2\dx\dt\ds\\
&+\frac1h\dint_{\nsk 0}^h\int_{\Omega}\int_{a+h-s}^{a+h}\unite u_t\Delta^s_t(u_t)\dt\dx-\int_{\Omega}\int_{T-a-h-s}^{T-a-h}\unite u_t \Delta^s_t(u_t)\dt\dx\ds
\\
&+\frac1{h} \int_{\Omega}\int_{a+h}^{T-a-h}\unite A(Du)(t)D\dint_{\nsk 0}^h 2u_t(t)-u_t(t+s)-u_t(t-s)\ds\dx\dt\\
&=\frac1h\dint_{\nsk 0}^h\int_{\Omega}\int_{a+h}^{T-a-h}\unite \Delta^s_t(f)\Delta^s_t(u_t)\dt\\
&+\int_{a+h-s}^{a+h}\unite f\Delta^s_t(u_t)\dt-\int_{T-a-h-s}^{T-a-h}\unite f\Delta^s_t(u_t)\dt\dx\ds=:(IV)
\end{align*}
We start by estimating $(II)$ with Young's inequality
\begin{align*}
 \abs{(II)}&\leq \dint_{\nsk 0}^h\dint_a^{a+h}\int_{\Omega}\abs{u_t(t)}^2+\abs{u_t(t+s)}^2\dt+\dint_{T-2h-a}^{T-a-h}\unite \abs{u_t(t)}^2+\abs{u_t(t+s)}^2\dt\dx\ds \\
&\leq c\sup_{(a,2a+h)\cup(T-2h-a,T-a)}\int_{\Omega}\abs{u_t}^2\dx.
\end{align*}
Next we estimate $(IV)$ by Young's inequality and similar as before
\begin{align*}
 \abs{(IV)}\leq &\delta (I)+ c_\delta\frac1h\dint_{\nsk 0}^h \int_{\Omega}\int_{a+h}^{T-a-h}\unite\abs{\Delta^s_t(f)}^2\dx\dt\ds+c\sup_{(a,2a+h)\cup(T-2h-a,T-a)}\int_{\Omega}\abs{u_t}^2+\abs{f}^2\dx.
\end{align*}
We divide $(III)$ into
\begin{align*}
 (III)&= \frac2{h}\int_{\Omega}\int_{a+h}^{T-a-h}\unite  A(Du)(t)Du_t(t)\dx\dt\\
&\quad-\frac1{h^2}\int_{\Omega}\int_{a+h}^{T-a-h}\unite  A(Du(t))(Du(t+h)-Du(t-h))\ds\dx\dt=(III)_1+(III)_2.
\end{align*}
\begin{align*}
 (III)_1=\frac2{h^2}\int_{\Omega}\int_{a+h}^{T-a-h}\unite  \partial_t F(Du(t))\dx\dt=\frac2{h^2}\int_\Omega F(Du(T-a-h)-F(Du(a+h))\dx
\end{align*}

To estimate $(III)_2$ we need to use Lemma~\ref{lem:t} 
\begin{align*}
 (III)_2&=  \frac{1}{h^2} \int_{\Omega}\int_{a+h}^{T-a-h}\unite\big(F(Du(t-h))-F(Du(t+h))\big)\\
&\quad -\frac{c_2}{2h^2} \abs{\Delta^{-h}_{t}V(Du)(t)}+\frac{c}{2h^2}\abs{\Delta^{h}_{t}V(Du)(t)}^2\dx\dt\\
&\geq-c\int_{\Omega}\int_{a+h}^{T-a-h} \Bigabs{\frac{\Delta^{h}_{t}V(Du)(t)}{h}}^2\dx\dt\\
&\quad+ \frac1{h^2}\Big(\int_a^{a+2h}\int_{\Omega}F(Du(s))\dx\ds -\int_{T-a-2h}^{T-a}\int_{\Omega}F(Du(s))\dx\ds\Big)\\
&= -c\int_{\Omega}\int_{a+h}^{T-a-h} \Bigabs{\frac{\Delta^{h}_{t}V(Du)(t)}{h}}^2\dx\dt\\
&\quad+  \frac2h \Big(\dint_a^{a+2h}\int_{\Omega}F(Du(s))\dx\ds -\dint_{T-a-2h}^{T-a}\int_{\Omega}F(Du(s))\dx\ds\Big)
\end{align*}
where we used \eqref{basic:1}.
This implies
\begin{align}
\label{eq:zwischen}
\begin{aligned}
 (I)\leq &c\abs{(II)}+c(\abs{(IV)}-\delta(I))+c\int_{\Omega}\int_{a+h}^{T-a-h} \Bigabs{\frac{\Delta^{h}_{t}V(Du)(t)}{h}}^2\dx\dt\\
&+\frac{2}h\dint_a^{a+2h}\int_{\Omega}F(Du(s))\dx\ds-F(Du(a+h))\dx\ds\\
& +\frac{2}h\dint_{T-a-2h}^{T-a}\int_{\Omega}F(Du(T-a-h))-F(Du(s))\dx\ds.
\end{aligned}
\end{align}
The terms in consideration have the factor $\frac{1}{h}$, which looks to much. However, by an  integration with respect to $a$, there is an additional cancellation effect which gives the ``correct'' order in $h$. This procedure is an important moment in the proof.

We integrate $a$ over the interval $[0,A]$ to find
\begin{align}
\label{eq:A1}
\begin{aligned}
\frac1h\int_0^A\dint_a^{a+2h}&\int_{\Omega}F(Du(s))\dx\ds-F(Du(a+h))\dx\ds\, da\\
&=\frac1h\int_{\Omega}\int_0^{A}\dint_{\nsk 0}^{2h}F(Du(a+s))\,da\ds-\int_h^{A+h}F(Du(a))\,da\dx\\
&=\frac1h\int_{\Omega}\dint_{\nsk 0}^{2h}\int_s^{A+s}F(Du(a))\,da\ds-\int_h^{A+h}F(Du(a))\,da\dx\\
&=\frac1{h}\int_{\Omega}\dint_{\nsk 0}^{2h}\int_{s}^{h}F(Du(a))\,da-\int_{A+s}^{A+h}F(Du(a))\,da\ds\dx\\
&=\int_{\Omega}\dint_{\nsk 0}^{2h}\frac{\abs{h-s}}{h}\Big(\dint_{s}^{h}F(Du(a))\,da-\dint_{A+s}^{A+h}F(Du(a))\,da\Big)\ds\dx\\
&\leq \sup_{(0,h)\cup(A+h,A+2h)}\int_{\Omega}F(Du)\dx.
\end{aligned}
\end{align}
The second difference of $F$ is estimated analogous, such that
\begin{align}
\label{eq:A2}
\begin{aligned}
 \int_0^A\frac1h&\dint_{T-a-2h}^{T-a}\int_{\Omega}F(Du(T-a-h))-F(Du(s))\dx\ds\, da\\
&=\frac1h\int_{\Omega}-\int_{T-A-h}^{T-h}F(Du(a))\,da+\dint_{\nsk 0}^{2h}\int_{T-A-2h+s}^{T-2h+s}F(Du(a))\,da\ds\dx\\
&= \frac1h\int_{\Omega}\dint_{\nsk 0}^{2h}-\int_{T-A-h}^{T-A-2h+s}F(Du(a))\,da
+ \int_{\Omega}\int_{T-h}^{T-2h+s}F(Du(a))\,da\ds\dx\\
&\leq  \sup_{(T-A-2h,T-A-h)\cup(T-h,T)}\int_{\Omega}F(Du)\dx.
\end{aligned}
\end{align}
Combining \eqref{eq:zwischen} with \eqref{eq:A1} and \eqref{eq:A2} gives
\begin{align*}
 & \dint_{\nsk 0}^h\int_{A+h}^{T-A-h}\int_{\Omega}\frac{\abs{\Delta^s_t(u_t)(t)}^2}{h}\dx\dt\ds
\leq  \dint_{\nsk 0}^A\int_0^h \int_{\Omega}\int_{a+h}^{T-a-h}\frac{\abs{\Delta^s_t(u_t)(t)}^2}{h}\dx\dt\ds\, da\\
&\leq \frac{c}{A}\bigg( \sup_{(0,A+2h)\cup (T-A-2h,T)}\int_{\Omega}F(Du)\dx
+\sup_{(0,2h+A)\cup(T-2h-A,T)}\int_{\Omega}\abs{u_t}^2\dx\\
&+\frac1h\dint_{\nsk 0}^h\int_{h}^{T-h}\int_{\Omega}\abs{\Delta^s_t(f)}^2\dx\dt\ds
+\sup_{(0,2h+A)\cup(T-2h-A,T)}\int_{\Omega}\abs{u_t}^2+\abs{f}^2\dx\\
&+c\int_{\Omega}\int_{h}^{T-h} \Bigabs{\frac{\Delta^{h}_{t}V(Du)(t)}{h}}^2\dx\dt\bigg)\leq c\frac{K}{A}
\end{align*}
by Sobolev embedding and \eqref{eq:timeall}.
\end{proof}
\begin{proof}[Proof of Theorem~\ref{thm:main}]
We imply Proposition~\ref{pro:N12t} on the interval on $[\frac{a}{2},T]$ for $A=\frac{a}2$. Lemma~\ref{lem:diening} then concludes the proof.
\end{proof}

\section{Estimates in space direction}
In this section we derive estimates for mixed derivatives (in time and space) for the whole space situation. Therefore $\Omega\equiv\setR^n$ and $Q_T=(0,T)\times\setR^n$.

The aim of this section is to show fractional derivatives in all directions including slanted ones in space-time. Naturally we will assume more regularity of the data in space direction. The following proposition is the space-analogue of Proposition~\ref{pro:Vt}. It can be derived by formally testing with $-\Delta u$
\begin{proposition}
\label{pro:DV}
Let Assumption~\ref{ass:phi} hold, $0<2a<T$ and $f\in L^{\phi^*}((0,T),(V^{1,\phi}(\Omega))^*)$. If additionally $\nabla f\in L^{\phi^*}(Q_T)$ we have for $u$ the solution of \eqref{eq} or \eqref{stokes}.
\[
 \int_a^T\int_{\setR^n}\abs{\nabla V(Du)}^2\dx\dt+\sup_{t\in[a,T]}\int_{{\setR^n}}\abs{\nabla u}^2\dx\leq c \int_0^T\int_{\setR^n}\phi^*(\abs{\nabla f})+\phi(\abs{\nabla u})\dx\dt.
\]
\end{proposition}
The proof of the above proposition for \eqref{eq} follows by the difference quotient technique. Basically the technique of \cite[Theorem~6.1]{DieE08} can be applied to the parabolic situation with general $F$. For the convenience of the reader we include a proof in the appendix.

 As before to simplify notation we shift the time interval by $A$, such that we can assume in the following
\begin{align}
 \label{eq:all}
\begin{aligned}
&\sum_{i=1}^n\int_0^T\int_{{\setR^n}}\abs{D^h_{x_i} V(Du)}^2+\abs{D^h_t V(Du)}^2+\abs{u_t}^2+\phi(\abs{Du})\dx\dt\\
& +\sup_{t\in[0,T]}\int_{{\setR^n}}\abs{\nabla u}^2+\abs{u_t}^2+F(Du)\dx\leq \frac{K}{A}.
\end{aligned}
\end{align}
for all $0<h\leq A$.
Where $K$ depends on $f$ by Proposition~\ref{pro:DV} and \eqref{eq:timeall}.
The next step is to redo the argument before in space directions. 
\begin{lemma}

\label{lem:ort}
 Let $u$ be a solution to \eqref{eq} or \eqref{stokes}, and let the assumptions on $f$ be such that \eqref{eq:all} holds. Then
%
\[
 \Bigabs{\frac1h\dint_{\nsk 0}^h\int_0^T\int_{\setR^n}\Delta^s_{x_i}(u_t)(x)\Delta^s_{x_i}(u_{x_i})(x)\dx\ds}\leq c\frac{K}{A},
\]
Where $\frac{K}{A}$ is the right hand side of \eqref{eq:all}. 
\end{lemma}
\begin{proof}
 We use the test function $\frac1h\dint_{\nsk 0}^h\Delta^{-s}_{x_i}\Delta^s_{x_i}(u_{x_i})(x)\ds$. 
This implies for almost every $t\in [0,T]$
\begin{align*}
 (I)+(II)+(III)&:=\frac1h\dint_{\nsk 0}^h\int_{{\setR^n}}\Delta^s_{x_i}(u_t  )(x)\Delta^s_{x_i}(u_{x_i})(x)\dx\ds\\
&+\frac1{h}\int_{{\setR^n}} A(Du)(x)\cdot D\dint_{\nsk 0}^h 2u_{x_i}(x)-u_{x_i}(x+se_i)-u_{x_i}(x-se_i)\ds (x)\dx\\
&=\frac1h\dint_{\nsk 0}^h\int_{{\setR^n}}\Delta^s_{x_i}(f)\cdot\Delta^s_{x_i}(u_{x_i})\dx\ds=:(IV)
\end{align*}

The term $(IV)$ can be estimated 
\begin{align*}
 \abs{(IV)}\leq  \dint_{\nsk 0}^h\int_{{\setR^n}}\phi^*(\abs{D^s_{x_i}f(x)})+\phi(\abs{Du})\dx.
\end{align*}
The terms $(II)$ can be estimated exactly as in the proof of Proposition~\ref{pro:N12t},
\begin{align*}
 (II)&=\frac1{h}\int_{{\setR^n}}A(Du(x))\cdot \partial_i(Du(x))2+\frac1{h^2}\int_{{\setR^n}}A(Du(x))\cdot (Du(x)-Du(x+he_i))\dx\\
&\quad+\frac2{h^2}\int_{{\setR^n}}A(Du(x))\cdot (Du(x-he_i)-Du(x))\dx\\
&=-\frac2{h}\int_{{\setR^n}}\partial_iF(Du(x))\dx-\frac1{h}\int_{{\setR^n}}A(Du(x))(D^{h}_{x_i}(Du(x))+D^{-h}_{x_i}(Du(x))\dx.
\end{align*}
The first term vanishes, the second can be estimated by Corollary~\ref{cor:schraeg}: 
\begin{align*}
\abs{(II)}&\leq -\frac1{h}\int_{{\setR^n}}(D^h_{x_i}F(Du(x))+D^{-h}_{x_i}F(Du)(x))\dx\\
&\quad +c\int_{{\setR^n}}\Bigabs{\frac{\Delta^{-h}_{x_i} V(Du)}{h}}^2\leq c \frac{K}{A}.
\end{align*}
\end{proof}
We can combine Proposition~\ref{pro:N12t} and the last lemma to get diagonal directions.
\begin{lemma}
\label{lem:mixed}
 Let $u$ be a solution to \eqref{eq} or\eqref{stokes}, and let the assumption on $f$ be such that \eqref{eq:all} holds. Then for every $h,A>0$ 
\begin{align*}
 & \Bigabs{\int_0^A\dint_{\nsk 0}^h\int_{a+h}^{T-a+h} 
\int_{\setR^n}\frac{\Delta^s_\slant(u_t)\cdot \Delta^s_{\slant} (\partial_\slant u)}{h}\dx\dt\ds\, da}
\leq \frac{K}{A}
\end{align*}
where $\frac{K}{A}$ depends (linearly) on the right hand side of \eqref{eq:all} by the characteristics of $F$.
\end{lemma}
\begin{proof}
We will use the testfunction
\[
 \frac1h\dint_{\nsk 0}^h\Delta^{-s}_\slant\Delta^s_\slant\partial_\slant u\ds,
\]
%
which is admissible by the same argument as in the proof of Proposition~\ref{pro:N12t}.
We integrate over $[a+h,T-a-h]\times {\setR^n}$ and find 
\begin{align*}
 &(I)+(II)+(III)\\
&:=\frac1h\dint_{\nsk 0}^h\int_{{\setR^n}}\int_{a+h}^{T-a-h}\unite (u_t(t,x))\Delta^{-s}\Delta^s_\slant(\partial_\slant u)\dx\dt\ds\\
&\quad+\frac2{h}\int_{{\setR^n}}\int_{a+h}^{T-a-h}\unite  A(Du)(t,x)\cdot D(\partial_\slant u)(t,x)\dx\dt\\
&\quad-\frac1{h}\int_{{\setR^n}}\int_{a+h}^{T-a-h}\unite  A(Du)(t,x)D\dint_{\nsk 0}^h \partial_\slant u(t+s,x+e_is)
 +  \partial_\slant u(t-s,x-e_is)\ds\dx\dt\\
&=\frac1h\dint_{\nsk 0}^h\int_{{\setR^n}}\int_{a+h}^{T-a-h}\unite f\Delta^{-s}_\slant\Delta^s_\slant(\partial_\slant u)\ds\dx\dt
=:(IV)
\end{align*}
We start by estimating $(I)$ with partial summation
\begin{align*}
(I)&= \frac1h\dint_{\nsk 0}^h\int_{{\setR^n}}\int_{a+h}^{T-a-h-s}\unite (\Delta^{s}_\slant u_t)(t,x)\Delta^s_\slant(\partial_\slant u)\dx\dt\ds\\
&\quad+\frac1h\dint_{\nsk 0}^h\int_{{\setR^n}}\int_{a+h-s}^{a+h}u_t(t+s,x+e_is)\Delta^s_\slant(\partial_\slant u)(t,x)\dt\dx-\int_{{\setR^n}}\int_{T-a-h-s}^{T-a-h}\unite u_t \Delta^s_\slant\partial_\slant u\dt\dx\ds\\
&=(I)_1+(I)_2-(I)_3.
\end{align*}
We find that 
\[
(I)_1+(II)+(III)= -(I)_2+(I)_3+(IV)
\]
We estimate 
\[
 \abs{-(I)_2+(I)_3}\leq c\sup_{t\in[a,a+2h]\cup[T-a-2h,T-a]}\int_{\setR^n} \abs{u_t}^2+\abs{Du}^2\dx,
\]
which can be estimated by \eqref{eq:all} (after taking the supremum of some terms over time).

We estimate $(IV)$ by
\begin{align*}
(IV)&=\frac1h\dint_{\nsk 0}^h\int_{{\setR^n}}\bigg(\int_{a+h}^{T-a-h-s}\unite \Delta^{s}_\slant f \Delta^s_\slant(\partial_\slant u)\dt\\ 
&\quad+\int_{a+h-s}^{a+h}f(t+s,x+e_i s)\Delta^s_\slant(u_t+u_{x_i})(t,x)\dt-\int_{T-a-h-s}^{T-a-h}\unite f(t,x)\Delta^s_\slant(u_t+u_{x_i})(t,x)\dt\bigg)\dx\ds,
\end{align*}
 where the terms depending on $u$ can again be estimated by \eqref{eq:all}.
This implies that
\begin{align}
\begin{aligned}
 \abs{(IV)}&\leq \int_{{\setR^n}}\dint_{\nsk 0}^h\int_{a+h}^{T-a-h-s}\unite \abs{D^s_\slant f}^2\ds+\abs{Du}^2+\abs{u_t}^2\dx\dt\\
&+2\int_{{\setR^n}}\dint_{a}^{a+2h}\abs{u_t}^2+\abs{Du}^2\dx\dt 
+ 2\int_{{\setR^n}}\dint_{a+h}^{a+2h}\abs{f}^2\dx\dt\\
&+2\int_{{\setR^n}}\dint_{T-a-2h}^{T-a}\abs{u_t}^2+\abs{Du}^2\dx\dt 
+ 2\int_{{\setR^n}}\dint_{T-a-2h}^{T-a-h} \abs{f}^2\dx\dt\\
&\leq \int_{{\setR^n}}\dint_{\nsk 0}^h\int_{a+h}^{T-a-h-s}\unite \abs{D^s_\slant f}^2\ds+\abs{Du}^2+\abs{u_t}^2\dx\dt\\
&+2\sup_{(a,a+2h)\cup(T-a-2h,T-a)}\int_{\setR^n}\abs{u_t}^2+\abs{Du}^2+\abs{f}^2\dx
\end{aligned}
\end{align}

The terms $(II)$ and $(III)$ will be estimated as in the analogous estimates before.
\begin{align*}
 (II)&=\frac{2}{h}\int_{{\setR^n}}\int_{a+h}^{T-a-h}\unite  \partial_\slant F(Du)\dx\dt\\
&= \frac{2}{h}\int_{{\setR^n}}F(Du(T-a-h))-F(Du(a+h))\dx
\end{align*}
The term $(III)$
we need to use Corollary~\ref{cor:schraeg} 
\begin{align*}
 (III)& =-\frac1{h^2}\int_{{\setR^n}}\int_{a+h}^{T-a-h}\unite  A(Du)(t,x)D (u(t+h,x+e_ih)-u(t-h,x-e_ih))\ds\dx\dt\\ \\
&\geq \frac{1}{h^2} \int_{{\setR^n}}\int_{a+h}^{T-a-h}\unite  \big(F(Du(t-h,x-e_ih))-F(Du(t+h,x+e_ih))\big)\\
&\quad -\frac{c_2}{h^2} \abs{\Delta^{-h}_\slant V(Du)(t)}+\frac{c}{h^2}\abs{\Delta^{h}_\slant V(Du)(t)}^2\dx\dt\\
&\geq-c\int_{a}^{T-a} \int_{{\setR^n}}\Bigabs{\frac{\Delta^{h}_\slant V(Du)}{h}}^2\dx\dt\\
&\quad+ \frac2h\dint_a^{a+2h}\int_{{\setR^n}}F(Du(s))\dx\ds -\frac2h\dint_{T-a-2h}^{T-a}\int_{{\setR^n}}F(Du(s))\dx\ds.
\end{align*}

Therefore, we integrate $a$ over $(0,A)$ and have
\begin{align*}
\abs{\int_0^A(I)_1\da}&\leq \int_0^A\abs{-(I)_2+(I)_3}+\abs{(IV)}+c\int_{a}^{T-a} \int_{{\setR^n}}\Bigabs{\frac{\Delta^{h}_\slant V(Du)}{h}}^2\dx\dt\da\\
 &\quad+\biggabs{\int_0^A\frac2h\dint_a^{a+2h}\int_{{\setR^n}}F(Du(s))-F(Du(a+h))\dx\ds\da}\\
 &\quad +\biggabs{\int_0^A\frac2h\dint_{T-a-2h}^{T-a}\int_{{\setR^n}}F(Du(s))-F(Du(T-a-h))\dx\ds\da}
\end{align*}
 By \eqref{eq:A1} and \eqref{eq:A2} we gain
\begin{align*}
\abs{\int_0^A(I)_1\da}&\leq c\sup_{t\in[0,A+2h]\cup[T-A-2h,T]}\int_{\setR^n} \abs{u_t}^2+\abs{Du}^2\dx\\
&+\int_{{\setR^n}}\dint_{\nsk 0}^h\int_{h}^{T-h-s}\abs{D^s_\slant f}^2\ds+\abs{Du}^2+\abs{u_t}^2\dx\ds\dt\\
&+2\sup_{(0,A+2h)\cup(T-A-2h,T)}\int_{\setR^n}\abs{u_t}^2+\abs{Du}^2+\abs{f}^2\dx\\
&+c\int_{0}^{T} \int_{{\setR^n}}\Bigabs{\frac{\Delta^{h}_\slant V(Du)}{h}}^2\dx\dt+\sup_{(0,A+2h)\cup (T-A-2h,T)}\int_{\Omega}\phi(\abs{Du})\dx.
\end{align*}
All terms can either be estimated by \eqref{eq:all} such that we gain the result.
\end{proof}
Theorem~\ref{thm:main2} follows from the following proposition 
and Lemma~\ref{lem:diening}.
\begin{proposition}
 \label{thm:alz} 
Let $F$ hold Assumption~\ref{ass:p} or Assumption~\ref{ass:phi} on the wholespace $\setR^n$. Let $u_0\in L^2(\setR^n)$ and $f\in L^{\phi^*}([0,T),(V^{1,\phi}_0(\setR^n))^*)$.
Let $u$ be a solution of \eqref{eq} or \eqref{stokes}. 

If additionally $f\in W^{1,2}((0,T),\Omega)$ and $\nabla f\in L^{\phi^*}(Q_T)$, then, 
\begin{align*}
 & \dint_{\nsk 0}^h\int_{A+h}^{T-A-h}\int_{\Omega}\frac{\abs{\Delta^s_{x_i}(u_t)(t)}^2}{\sqrt{h}}\dx\dt\ds
\leq \frac{K}{A},
\end{align*}
for every $A\geq h>0$. 
The constant $K$ only depends on the characteristics of $F$ and the regularity assumptions of $f$.
\end{proposition}
\begin{proof}
We start with a few basic calculations. We reformulate
\begin{align*}
\Delta^s_{x_i}u_t&= u_t(t,x+e_is)-u_t(t,x)=u_t(t+s,x+e_is)-u_t(t,x) + u_t(t,x+se_i)-u_t(t+s,x+se_i)\\
&= \Delta^s_\slant(u_t)(t,x)+\Delta^s_tu_t(x+se_i).
\end{align*}

Therefore, we find
\begin{align*}
\abs{\Delta^s_{x_i}u_t}
&\leq \abs{ \Delta^s_\slant(u_t)(t,x)}+\abs{\Delta^s_tu_t(x+se_i)}.
\end{align*}
The last term can eventually be estimated by Proposition~\ref{pro:N12t}. Therefore, once we can estimate the "Queer Quotient", we gain the result.
He is estimated by
\[
 \abs{\Delta^s_\slant(u_t)(t,x)}^2 =\Delta^s_\slant(u_t)(t,x)\cdot\Delta^s_\slant(u_t+u_{x_i})(t,x)-\Delta^s_\slant(u_t)(t,x)\cdot\Delta^s_\slant u_{x_i}(t,x)= S_0-S_1
\]
The term $S_0$ can be estimated by Lemma~\ref{lem:mixed}.

We analyse $S_1$ 
\begin{align}
\label{eq:Mterms}
\begin{aligned}
 S_1&=\Delta^s_{x_i} (u_t)(t,x)\cdot \Delta^s_{x_i} (u_{x_i})+\Delta^s_{x_i} (u_t)(t,x)\cdot \Delta^s_{t} (u_{x_i})\\
&\quad+\Delta^s_{t} (u_t)(t,x)\cdot \Delta^s_{x_i} (u_{x_i})+\Delta^s_{t} (u_t)(t,x)\cdot \Delta^s_{t} (u_{x_i})\\
&:= M_1+M_2+M_3+M_4
\end{aligned}
\end{align}
This implies  by Proposition~\ref{pro:N12t} and Lemma~\ref{lem:mixed}
\begin{align*}
  \dint_{\nsk 0}^h&\int_{A+h}^{T-A-h}\int_{\setR^n}\abs{\Delta^s_{x_i}(u_t)(t)}^2\dx\dt\ds\\
&\leq \dint_{\nsk 0}^A\dint_{\nsk 0}^h \int_{\setR^n}\int_{a+h}^{T-a-h}\unite\abs{\Delta^s_{x_i}(u_t)(t)}^2\dx\dt\ds\da\\
&\leq \dint_{\nsk 0}^A\dint_{\nsk 0}^h\int_{A+h}^{T-A-h}\int_{\setR^n}\abs{\Delta^s_{t}(u_t)(t)}^2\dx\dt\ds\\
&\quad + \Bigabs{\int_{\setR^n}\dint_{\nsk 0}^A\dint_{\nsk 0}^h\int_{a+h}^{T-a-h}\unite  
\Delta^s_\slant(u_t)(t)\cdot \Delta^s_{\slant} (\partial_\slant u)\dx\dt\ds\, da}\\
&\quad +  \Bigabs{\int_{\setR^n}\dint_{\nsk 0}^A\dint_{\nsk 0}^h
\int_{a+h}^{T-a-h}\unite  M_1+M_2+M_3+M_4\dx\dt\ds\, da}\\
&\leq \frac{hK}{A} +  \Bigabs{\int_{\setR^n}\dint_{\nsk 0}^A\dint_{\nsk 0}^h\int_{a+h}^{T-a-h}\unite M_1+M_2\dx\dt\ds\, da} +\dint_{\nsk 0}^h\int_{h}^{T-h} 
\int_{\setR^n}\abs{M_1+M_2}\dx\dt\ds.
\end{align*}
$M_1$ can be estimated by Lemma~\ref{lem:ort}.
\begin{align}
\label{befor}
 \Bigabs{\dint_{\nsk 0}^A\dint_{\nsk 0}^h\int_{a-h}^{T-a-h}\unite M_1\dx\dt\ds\, da}\leq \frac{Kh}{A}.
\end{align}

We are left to estimate the integrals with $M_2,M_3$ and $M_4$. Unlike the other terms they can only be estimated by less, i.e. by $\sqrt{h}K/A$.
We estimate the integral of $M_3$ with H\"older and Young. 
\begin{align*}
\frac1{\sqrt{h}} &\int_{\setR^n}\int_0^{T-h}\dint_{\nsk 0}^h \abs{\Delta^s_{t} (u_t)(t,x)\cdot \Delta^s_{x_i} (u_{x_i})}\ds\dt\dx\\
&\leq c\sup_{[0,T]}\int_{\setR^n}\abs{Du}^2\dx+c\int_{\setR^n}\int_0^{T-h}\dint_{\nsk 0}^h \frac{\abs{\Delta^s_{t} (u_t)(t,x)}^2}{h}\ds\dt\dx.
\end{align*}
The integral with $M_4$ can be estimated analogous:
\begin{align*}
\frac1{\sqrt{h}} &\int_{\setR^n}\int_0^{T-h}\dint_{\nsk 0}^h \abs{\Delta^s_{t} (u_t)(t,x)\cdot \Delta^s_{t} (u_{x_i})}\ds\dt\dx\\
&\leq c\sup_{[0,T]}\int_{\setR^n}\abs{Du}^2\dx+c\int_{\setR^n}\int_0^{T-h}\dint_{\nsk 0}^h \frac{\abs{\Delta^s_{t} (u_t)(t,x)}^2}{h}\ds\dt\dx.
\end{align*}

To estimate $M_2$ we use partial summation and proceed as before
\begin{align}
\label{eq:parsum}
\begin{aligned}
 &\biggabs{\frac1{\sqrt{h}}\dint_{\nsk 0}^A\dint_{\nsk 0}^h \int_{\setR^n}\int_{a+h}^{T-a-h}\unite \Delta^s_{x_i} (u_t)(t,x)\cdot \Delta^s_{t} (u_{x_i})\dt\ds\dx\da}\\
 &\leq\biggabs{\frac1{\sqrt{h}}\dint_{\nsk 0}^A  \int_{\setR^n}\dint_{\nsk 0}^h \int_{a+h+s}^{T-a-h}\unite  \Delta^{-s}_{t} (u_t)(t,x)\cdot \Delta^{-s}_{x_i} u_{x_i}\dt\ds\dx\da}\\
 &\quad+\biggabs{\frac1{\sqrt{h}}\dint_{\nsk 0}^A \dint_{\nsk 0}^h \int_{\setR^n}\int_{T-a-h}^{T-a-s}\!\!\!\!\!\!\!\!\!\!\!\!\! \Delta^{s}_{x_i} (u_t)(t,x)\cdot(u_{x_i})\dt
 - \int_{a+h}^{a+h+s} \Delta^{s}_{x_i} (u_t)(t,x)\cdot(u_{x_i})\dt\ds\dx\da }\\&=:(I)+(II)
 \end{aligned}
\end{align}
The estimate on $(II)$ can be estimated optimally
 \begin{align}
 \label{eq:2opt}
 (II) &\leq 2\sqrt{h}\sup_{[0,2h]\cap[T-2h,T]}\int_{\setR^n}\abs{Du}^2+\abs{u_t}^2\dx.
 \end{align}
 The term $(I)$ is again estimated by
 \[
 (I)\leq c\sup_{[0,T]}\int_{\setR^n}\abs{Du}^2\dx+c\int_{\setR^n}\int_0^{T-h}\dint_{\nsk 0}^h \frac{\abs{\Delta^s_{t} (u_t)(t,x)}^2}{h}\ds\dt\dx.
 \]
\end{proof}
Next we will prove Theorem~\ref{cor:alz}. In the following we assume power-law structure. I.e. that Assumption~\ref{ass:p} hold for \eqref{eq} which is assumed to be in terms of full gradients (i.e. $Du\equiv \nabla u$).

\begin{proof}[Proof of Theorem~\ref{cor:alz}]
We will use the following Sobolev embedding
\[
 \int_0^T\int_{\setR^n}\abs{V(Du)}^\frac{2(n+1)}{n-1}\dx\dt\leq \bigg(\int_0^T\int_{\setR^n}\abs{\nabla V(Du)}^2+\abs{\partial_t V(Du)}^2\dx\dt\bigg)^\frac{n+1}{n-1}<\infty.
\]
We will also use the following estimate in case $n\geq 3$. 
\begin{align*}
 \int_0^T\int_{\setR^n}\abs{\nabla u}^{p+\frac{4}{n}}\dx\dt&\leq \int_0^T\int_{\setR^n}\bigg(\abs{\nabla u}^{p\frac{n}{n-2}}\dx\bigg)^\frac{n-2}{n}\int_{\setR^n}\bigg(\abs{\nabla u}^{\frac{4}{n}\frac{n}{2}}\dx\bigg)^\frac{2}{n}\dt\\
&=\int_0^T\bigg(\int_{\setR^n}\abs{V(D u)}^{\frac{2n}{n-2}}\dx\bigg)^\frac{n-2}{n}\dt\sup_{t\in (0,T)}\bigg(\int_{\setR^n}\abs{\nabla u}^{2}\dx\bigg)^\frac{2}{n}\\
&\leq\int_0^T\int_{\setR^n}\abs{\nabla V(D u)}^{2}\dx\dt\sup_{t\in (0,T)}\bigg(\int_{\setR^n}\abs{\nabla u}^{2}\dx\bigg)^\frac{2}{n}
\end{align*}
Where we used Sobolev embedding. (By Korn's inequality the same estimates hold also in case of symmetric gradients).
Observe, that the above estimate implies $\nabla u\in L^q(Q_T)$, where $q:=\max\bigset{\frac{p(n+1)}{n-1},p+\frac{4}{n},2}$.
If $1<p\leq 2-\frac{4}{n}$ the best integrability of $\nabla u$ will be $2$, and Theorem~\ref{thm:main2} can not be improved.
In case $p>\min{2-\frac{4}{n},1}$ we enter the proof of Proposition~\ref{thm:alz} below \eqref{befor}. As mentioned there
we only need to estimate the terms involving $M_2,M_3,M_4$ (defined in \eqref{eq:Mterms}, as all other terms can be estimated by $hK/A$.

 We start by estimating $M_4$. We define $2\beta=\min\set{\frac12+\theta,1}$, where $\theta$ will be fixed at the end of the proof. Young's inequality implies.
\begin{align*}
 \frac{1}{h^{2\beta}}\abs{\Delta^s_{t} (u_t)\cdot \Delta^s_{t} (u_{x_i})}
&\leq \frac1{2h}\abs{\Delta^s_{t} (u_t)}^2+\frac1{2h^{2\theta}}\abs{\Delta^s_{t} (u_{x_i})}^2
\end{align*}
The first term can be estimated by Proposition~\ref{pro:N12t}. The second term is a mixed derivative.

\begin{enumerate}
 \item Case $p\in(\max\set{1,2-\frac{4}{n}},2]$. By Young's inequality and the definition of $V_i(Du)$ 
\begin{align*}
&\frac1{h^{2\theta}} \abs{\Delta^s_{t} (u_{x_i})(t)}^2
\leq \abs{\Delta^s_{t} (u_{x_i})(t)}^{2-2\theta }
\Big( \frac{(\abs{\nabla u(t+s)}+\abs{\nabla u(t)})^{p-2}}{(\abs{\nabla u(t+s)}+\abs{\nabla u(t)})^{p-2}}
\frac{\abs{\Delta^s_{t} (u_{x_i})(t)}^2}{h^2}\Big)^\theta\\
&\leq (\abs{\nabla u(t)}+\abs{\nabla u(t+s)})^\frac{2-2\theta}{1-\theta}
(\abs{\nabla u(t+s)}+\abs{\nabla u(t)})^\frac{(2-p)\theta}{1-\theta} +  c\abs{D_t^s V_i(\nabla u)(t)}^2= (I)+(II).
\end{align*}
The function $(II)$ is integrable. By Young's inequality we find 
\[
 (I)\leq c\abs{\nabla u(t)}^\frac{2-\theta p}{1-\theta}+c\abs{\nabla u(t+s)}^\frac{2-\theta p}{1-\theta}
\]
which is integrable if $\frac{2-\theta p}{1-\theta}=q$. Calculations imply, that we can choose
\[
 \theta=\max\Bigset{\frac{p+2-(2-p)n}{2p},1-\frac{n}4(2-p)}.
\]
Observe, that $\theta\geq \frac12$ for all $p\in [\max\set{1,2-\frac{4}{n}},2]$. As $\beta=\min\set{\frac12,\frac14+\frac\theta2}$, we have that $\beta=\frac12$ for $p\in [2-\frac{2}{n},2]$. 
\item Case $p\in [2,\infty)$ we fix $\theta=\frac{2}{p}$ and estimate  
\begin{align*}
&\frac1{h^{2\theta}} \abs{\Delta^s_{t} (u_{x_i})(t)}^2
\leq \abs{\Delta^s_{t} (u_{x_i})(t)}^{2-2\theta-\theta(p-2)}
\Big( (\abs{\nabla u(t+s)}+\abs{\nabla u(t)})^{p-2}
\frac{\abs{\Delta^s_{t} (u_{x_i})(t)}^2}{h^2}\Big)^\theta\\
&=  c\abs{D_t^s V_i(\nabla u)(t)}^{2\theta}
\end{align*}
which is locally integrable. If $u$ is space periodic, so is $u_t$ and for functions of this type local integrable functions are also globally integrable.
This implies that $\beta=\frac12$ for $p\in [2,4]$ if $u\in V^{1,p}_{\text{per}}(\setR^n)$.
\end{enumerate}
The term $\abs{M_3}$ can be estimated in the very same way. On the term $M_2$ we have to apply the partial summation as is done in \eqref{eq:parsum}; there the term $(II)$ can be estimated optimally (see \eqref{eq:2opt} and $(I)$ analogous to the term $M_4$ above.
\end{proof}
\begin{remark} We will give some explanation why large $p$ decrease the differentiability of $u_t$. Let us consider the model case \eqref{eq:plap}.
If $p>2$, the quantity $V(Du)$ is differentiable and not $Du$ itself. By \eqref{eq:hammer2} one observes that 
\[
 \int \abs{\partial_tV(Du)}^2 \sim \int \abs{Du}^{p-2}\abs{\partial_t Du}^2<\infty.
\]
Therefore,  $ \partial_t Du$ is only integrable with respect to the weight $\abs{Du}^{p-2}$. Theorem~\ref{cor:alz} says that $\beta\to \frac14$ for $p\to\infty$, which would be Proposition~\ref{thm:alz} again. We conjecture, that for general degenerate $F$ when $p\to\infty$, the best possible order of mixed derivative is $\partial^\frac14_{x_i} u_t$.
\end{remark}
\begin{remark}
 \label{rem:fracsym} 
Let us discuss the case of symmetric gradients in the framework of Theorem~\ref{cor:alz}.  
Only at the estimates of $M_1,...,M_4$ we needed the restriction. This is due to the fact that
$\nabla V(\nabla u)$ does not necessarily exist. Only $\nabla V(\bfepsilon u)$ is an $L^2$-function. Therefore, by the proof above it is only possible to get analog bounds on 
\[
\Bigabs{\frac{\Delta^s_{x_i}(u_t^j)+\Delta^s_{x_j}(u_t^i)}{h^\beta}}^2
\]
in the framework of symmetric gradients.
\end{remark}

\section{Appendix}
\begin{proof}[Proof of Proposition~\ref{pro:Vt}]
 We briefly recall that \eqref{eq} and \eqref{stokes} have a unique solution with additional regularity.

A simple and standard way to prove this is based on Ritz-Galerkin Approximations.
\[
 u_m(t,x)=\sum_{j=1}^m c_j^m(t)\psi_j(x), 
\]
with smooth Ansatz functions $(\phi_j)_{j\in \setN}\in W^{1,\phi}_0(\Omega)\cap L^2(\Omega)$; which are linear independent and whose linear hull is dense in $W^{1,\phi}_0(\Omega)\cap L^2(\Omega)$. We choose 
\[X_m= \set{v\in W^{1,\phi}_0(\Omega)\cap L^2(\Omega)| v= \sum_{j=1}^m c_j\psi_j(x)}\] 
and $ u^m_0\in X_m$, such that $u^m_0\to u_0$ strongly in $L^2(\Omega)$. The approximate equation reads
\[
 \skp{(u_m'(t)}{\psi_k}+\skp{A(\cdot,\nabla u_m)}{\psi_k}=\skp{f}{\phi_k},\,u_m(0)=u_0^m
\]
for all $1\leq k\leq m$ and is equivalent to a system of ODE for the $c_j^m$.

In the existence interval $[0,\tau)$ we may use the test function $u_m'$ and obtain a uniform estimate
\[
 \int_0^\tau\int_\Omega \abs{u'_m}^2\dx\dt + \int_\Omega F(Du_m)dx|^\tau_0\leq \int_0^\tau\int_\Omega\abs{f}^2\dx\dt
\]
 uniformly in $m\in\setN$. Hence the solution can be extended to the full interval $[0,T]$. Moreover, the following estimate holds for a. e. $\tau\in (a,T]$.
\begin{align}
\label{eq:1}
  \int_a^T\int_\Omega \abs{u'_m}^2\dx\dt + \sup_{\tau\in [a,T]}\int_\Omega F(Du_m)(\tau)x\leq \int_0^T\int_\Omega\abs{f}^2\dx\dt +\dashint_0^a\int_\Omega F(Du_m)dx dt.
\end{align}
Due to convexity the Ritz-Galerkin approximation is equivalent to the minimum problem
\begin{align}
\label{eq:minim}
 \int_0^\tau\int_\Omega u'_m(u_m-v) \dx\dt +\int_0^\tau\int_\Omega F(Du_m)\dx\dt\leq \int_0^\tau\int_\Omega F(Dv)+f(u_m-v)\dx\dt,
\end{align}
for all $v\in X_m$ and all $\tau\in [0,T]$.

By weak compactness we find a subsequence $u_m\weakto u$ and $u_m'\weakto u'$ in $L^2$ and $\nabla u_m\weakto \nabla u$ in $L^\phi$. In \eqref{eq:minim} we perform an integration with respect to $\tau\in [t_1,t_1+h]$ and $h$ small, then \eqref{eq:minim} can be rewritten as
\begin{align*}
 &\frac{1}{2}\dashint_{t_1}^{t_1+h}\int_\Omega \abs{u_m}^2 \dx-\int_0^\tau\int_\Omega u'_m v\dx\dt  +\int_0^\tau\int_\Omega F(Du_m)\dx\dt d\tau\\
&\quad \leq \dashint_{t_1}^{t_1+h}\int_0^\tau\int_\Omega F(Dv)+f(u_m-v)\dx\dt d\tau+\int_\Omega \abs{u_0^m}^2 \dx.
\end{align*}
Now we may pass to the limit $m\to \infty$ and due to lower semi continuity in the weak topology we find 
\begin{align}
\begin{aligned}
 &\frac{1}{2}\int_\Omega \abs{u}^2(\tau) \dx-\int_0^\tau\int_\Omega u' v\dx\dt  +\int_0^\tau\int_\Omega F(Du)\dx\dt d\tau\\
 &\quad\leq \int_0^\tau\int_\Omega F(Dv)+f(u_m-v)\dx\dt d\tau+\int_\Omega \abs{u_0}^2 \dx.
\end{aligned}
\end{align}
for a.e. $\tau$ and $v\in L^\phi((0,T), V^{1,\phi}_0(\Omega)\cap L^\infty((0,T),L^2(\Omega)$. This can be reformulated into
\begin{align*}
  \int_0^\tau\int_\Omega u_t(u-v) \dx\dt +\int_0^\tau\int_\Omega F(Du)\dx\dt\leq \int_0^\tau\int_\Omega F(Dv)+f(u_m-v)\dx\dt
\end{align*}
and a weak solution is constructed. Additionally, \eqref{eq:1} holds in the limit and implies the first estimate in Proposition~\ref{pro:Vt} (by Assumption~\ref{ass:phi}). 

We proceed by taking a sequence of $v_m\in X_m$, such that $\nabla v_m\to \nabla u$ strongly. By taking $v_m$ as a testfunction in \eqref{eq:minim} and passing to the limit we find.
\[
 \limsup_m\int_0^\tau\int_\Omega F(Du_m)\dx\dt d\tau\leq \int_0^\tau\int_\Omega F(Du)\dx\dt d\tau.
\]
This implies (by strict convexity)
\[
 \lim_m\int_0^\tau\int_\Omega F(Du_m)\dx\dt d\tau = \int_0^\tau\int_\Omega F(Du)\dx\dt d\tau\text{ and therefore  }D u_m\to Du\text{ strongly}.
\]
Furthermore, differentiating the Ritz Galerkin equations with respect to $t$ and testing with $u_m'$ we obtain by (e) of Assumption~\ref{ass:phi}
\begin{align}
\label{eq:estdis}
 \frac12\int_\Omega\abs{u'_m}^2\dx|^{\tau}_a+\int_a^\tau\int_\Omega\abs{\partial_tV(\nabla u_m)}^2\dx\dt\leq \int_a^\tau\int_\Omega \abs{f_t}^2+\abs{u'_m}^2\dx\dt.
\end{align}
We gain a subsequence $\partial_t V(Du_m)\weakto\partial_t \overline{V}$. However, by the strong convergence of $Du_m$ we find $\overline{V}=V(Du)$. Now passing to the limit in \eqref{eq:estdis} implies the desired estimates after  integration over $a$.
\end{proof}
\begin{proof}[Proof of Proposition~\ref{pro:DV}]
For the sake of simplicity we take $F(x,Q)=F(Q)$ omitting the lower order terms. By Assumption~\ref{ass:phi}(c) we have that
\[
 D^h_{x_i}(A(Du))\cdot D^h_{x_i}(Du)\sim \abs{D^h_{x_i}V(Du)}^2.
\]
We take the testfunction $-D^h_{x_i}(D_{x_i}^h(u))$, this implies by partial summation and Young's inequality \eqref{eq:young}
\[
 \int_0^\tau\int_{\setR^n} \frac12\partial_t \abs{D^h_{x_i}u}^2 +\abs{D^h_{x_i}V(Du)}^2\dx\dt\leq  \int_0^\tau\int_{\setR^n}\phi^*(\abs{D^h_{x_i}f})+\phi(\abs{D_{x_i}^h(u)}).
\]
By passing with $h\to0$ we get the desired estimate. 
\end{proof}
\bibliographystyle{abbrv}



\begin{thebibliography}{10}

\bibitem{AceMin07}
E.~Acerbi and G.~Mingione.
\newblock Gradient estimates for a class of parabolic systems.
\newblock {\em Duke Math. J.}, 136(2):285--320, 2007.

\bibitem{Ada75}
R.~A. Adams.
\newblock {\em {S}obolev spaces}.
\newblock Academic Press [A subsidiary of Harcourt Brace Jovanovich,
  Publishers], New York-London, 1975.
\newblock Pure and Applied Mathematics, Vol.~65.

\bibitem{BulEttKap10}
M.~Bul{\'{\i}}{\v{c}}ek, F.~Ettwein, P.~Kaplick{\'y}, and D.~Pra{\v{z}}{\'a}k.
\newblock On uniqueness and time regularity of flows of power-law like
  non-{N}ewtonian fluids.
\newblock {\em Math. Methods Appl. Sci.}, 33(16):1995--2010, 2010.

\bibitem{Bur14}
J.~Burczak.
\newblock Almost everywhere {H}\"older continuity of gradients to non-diagonal
  parabolic systems.
\newblock {\em Manuscripta Math.}, 144(1-2):51--90, 2014.

\bibitem{BurThe}
J.~Burczak.
\newblock Regularity of nonlinear non-diagonal evolutionary systems.
\newblock {\em PhD-thesis, submitted}, 2014.

\bibitem{DiB93}
E.~DiBenedetto.
\newblock {\em Degenerate parabolic equations}.
\newblock Springer-Verlag, New York, 1993.

\bibitem{DiBFri85}
E.~DiBenedetto and A.~Friedman.
\newblock H\"older estimates for nonlinear degenerate parabolic systems.
\newblock {\em J. Reine Angew. Math.}, 357:1--22, 1985.

\bibitem{DieE08}
L.~Diening and F.~Ettwein.
\newblock Fractional estimates for non-differentiable elliptic systems with
  general growth.
\newblock {\em Forum Mathematicum}, 20(3):523--556, 2008.

\bibitem{DieKap12}
L.~Diening and P.~Kaplick\'y.
\newblock $l^q$ theory for a generalized stokes system.
\newblock {\em Manuscripta Mathematica}, 141:333--361, 2013.

\bibitem{DieKapSch11}
L.~Diening, P.~Kaplick{\'y}, and S.~Schwarzacher.
\newblock B{MO} estimates for the {$p$}-{L}aplacian.
\newblock {\em Nonlinear Anal.}, 75(2):637--650, 2012.

\bibitem{DieKapSch13}
L.~Diening, P.~Kaplick{\'y}, and S.~Schwarzacher.
\newblock Campanato estimates for the generalized stokes system.
\newblock {\em Annali di Matematica Pura ed Applicata}, 2013.

\bibitem{DieSV09}
L.~Diening, B.~Stroffolini, and A.~Verde.
\newblock Everywhere regularity of functionals with {$\phi$}-growth.
\newblock {\em Manuscripta Math.}, 129(4):449--481, 2009.

\bibitem{FreSpe12a}
J.~Frehse and M.~Specovius-Neugebauer.
\newblock Fractional differentiability for the stress velocities to the
  solution of the {P}randtl-{R}euss problem.
\newblock {\em ZAMM Z. Angew. Math. Mech.}, 92(2):113--123, 2012.

\bibitem{FreSpe12b}
J.~Frehse and M.~Specovius-Neugebauer.
\newblock Fractional interior differentiability of the stress velocities to
  elastic plastic problems with hardening.
\newblock {\em Boll. Unione Mat. Ital. (9)}, 5(3):469--494, 2012.

\bibitem{Iwa83}
T.~Iwaniec.
\newblock Projections onto gradient fields and {$L^{p}$}-estimates for
  degenerated elliptic operators.
\newblock {\em Studia Math.}, 75(3):293--312, 1983.

\bibitem{Kap05}
P.~Kaplick{\'y}.
\newblock Time regularity of flows of non-{N}ewtonian fluids.
\newblock {\em IASME Trans.}, 2(7):1232--1236, 2005.

\bibitem{KapMalSta02}
P.~Kaplick{\'y}, J.~M{\'a}lek, and J.~Star{\'a}.
\newblock Global-in-time {H}\"older continuity of the velocity gradients for
  fluids with shear-dependent viscosities.
\newblock {\em NoDEA Nonlinear Differential Equations Appl.}, 9(2):175--195,
  2002.

\bibitem{KuuMin13}
T.~Kuusi and G.~Mingione.
\newblock Linear potentials in nonlinear potential theory.
\newblock {\em Arch. Ration. Mech. Anal.}, 207(1):215--246, 2013.

\bibitem{KuuMin123}
T.~Kuusi and G.~Mingione.
\newblock The {W}olff gradient bound for degenerate parabolic equations.
\newblock {\em J. Eur. Math. Soc. (JEMS)}, 16(4):835--892, 2014.

\bibitem{Lad67}
O.~Ladyzhenskaya.
\newblock New equations for the description of motion of viscous incompressible
  fluids and solvability in the large of boundary value problems for them.
\newblock {\em Proc. Stek. Inst. Math.}, 102:95--118, 1967.

\bibitem{Lad68}
O.~Ladyzhenskaya.
\newblock Modifications of the {N}avier-{S}tokes equations for large gradients
  of the velocities.
\newblock {\em Zap. Nau\v cn. Sem. Leningrad. Otdel. Mat. Inst. Steklov.
  (LOMI)}, 7:126--154, 1968.

\bibitem{Lin12}
P.~Lindqvist.
\newblock On the time derivative in an obstacle problem.
\newblock {\em Rev. Mat. Iberoam.}, 28(2):577--590, 2012.

\bibitem{Lio69}
J.-L. Lions.
\newblock {\em Quelques m\'ethodes de r\'esolution des probl\`emes aux limites
  non lin\'eaires}.
\newblock Dunod, 1969.

\bibitem{MalR05}
J.~M{\'a}lek and K.~R. Rajagopal.
\newblock Mathematical issues concerning the {N}avier--{S}tokes equations and
  some of its generalizations.
\newblock In {\em Evolutionary Equations}, volume~2 of {\em Handbook of
  Differential Equations}, pages 371--459.

\bibitem{Min06}
G.~Mingione.
\newblock Regularity of minima: an invitation to the dark side of the calculus
  of variations.
\newblock {\em Appl. Math.}, 51(4):355--426, 2006.

\bibitem{Mis02}
M.~Misawa.
\newblock Local {H}\"older regularity of gradients for evolutional
  {$p$}-{L}aplacian systems.
\newblock {\em Ann. Mat. Pura Appl. (4)}, 181(4):389--405, 2002.

\bibitem{Raj92}
K.~R. Rajagopal.
\newblock Mechanics of non-{N}ewtonian fluids.
\newblock In {\em Recent developments in theoretical fluid mechanics ({P}aseky,
  1992)}, volume 291 of {\em Pitman Res. Notes Math. Ser.}, pages 129--162.

\bibitem{Sch13}
S.~Schwarzacher.
\newblock H\"{o}lder-zygmund estimates for degenerate parabolic systems.
\newblock {\em Journal of Differential Equations}, 256:2423--2448, 2014.

\bibitem{Tri92}
H.~Triebel.
\newblock {\em Theory of function spaces. {II}}, volume~84 of {\em Monographs
  in Mathematics}.
\newblock Birkh\"auser Verlag, Basel, 1992.

\bibitem{Uhl77}
K.~Uhlenbeck.
\newblock Regularity for a class of non-linear elliptic systems.
\newblock {\em Acta Math.}, 138(3-4):219--240, 1977.

\bibitem{Ura68}
N.~N. Uralceva.
\newblock Degenerate quasilinear elliptic systems.
\newblock {\em Zap. Nau\v cn. Sem. Leningrad. Otdel. Mat. Inst. Steklov.
  (LOMI)}, 7:184--222, 1968.

\end{thebibliography}
\end{document}